\documentclass[12pt]{amsart}
\usepackage[english]{babel}
\usepackage{amsmath}
\usepackage{amsthm}
\usepackage{anysize}
\usepackage{mathtools}
\usepackage{mathrsfs}
\usepackage{amscd}
\usepackage{amsfonts}
\usepackage{amssymb}
\usepackage{amssymb}
\usepackage{xcolor}
\usepackage{xfrac}
\usepackage{esint}
\usepackage{graphicx}
\usepackage{float}
\usepackage{array}
\usepackage{enumitem}
\usepackage{tikz-cd}
\usepackage{centernot}
\usepackage{stmaryrd}
\usepackage[symbol]{footmisc}
\usepackage{comment}
\excludecomment{confidential}

\numberwithin{equation}{section}

\newtheorem{theorem}{Theorem}[section]
\newtheorem{corollary}[theorem]{Corollary}
\newtheorem{lemma}[theorem]{Lemma}
\newtheorem{proposition}[theorem]{Proposition}
\newtheorem{definition}[theorem]{Definition}

\DeclareMathOperator{\Alg}{Alg}
\DeclareMathOperator{\Ker}{Ker}

\DeclareMathOperator{\Spec}{Spec}

\begin{document}
	\title{Axiomatization of the degree of Fitzpatrick,  Pejsachowicz and  Rabier}
	\author{Juli\'an L\'opez-G\'omez, Juan Carlos Sampedro}
	\thanks{The authors have been supported by the Research Grant PGC2018-097104-B-I00 of the Spanish Ministry of Science, Technology and Universities and by the Institute of Interdisciplinar Mathematics of Complutense University. The second author has been also supported by PhD Grant PRE2019\_1\_0220 of the Basque Country Government.}
	
	\address{Institute of Interdisciplinary Mathematics (IMI) \\
Department of Analysis and Applied Mathematics \\
		Complutense University of Madrid \\
		28040-Madrid \\
		Spain.}
	\email{Lopez$\_$Gomez@mat.ucm.es, juancsam@ucm.es}

\begin{abstract}
In this paper we prove an analogue of the uniqueness theorems of F\"uhrer \cite{Fu} and Amann and Weiss \cite{AW} to cover the degree of Fredholm operators of index zero of Fitzpatrick,  Pejsachowicz and  Rabier \cite{FPR}, whose range of applicability is substantially wider than for the most classical degrees of
Brouwer \cite{Br} and Leray--Schauder \cite{LS}. A crucial step towards the axiomatization of
the Fitzpatrick--Pejsachowicz--Rabier degree is provided by the generalized algebraic multiplicity of Esquinas and L\'{o}pez-G\'{o}mez \cite{ELG,Es,LG01}, $\chi$, and the axiomatization theorem
of Mora-Corral \cite{MC,LGMC}. The latest result facilitates  the axiomatization of the parity of Fitzpatrick and Pejsachowicz \cite{FP2}, $\sigma(\cdot,[a,b])$, which provides  the key step for establishing
the uniqueness of the degree for Fredholm maps.
\end{abstract}
	
	\keywords{Degree for Fredholm maps, uniqueness, axiomatization, normalization, generalized
additivity, homotopy invariance, generalized algebraic multiplicity, parity, orientability. }
	\subjclass[2020]{ 47H11, 47A53, 55M25.}
	
		\maketitle

\section{Introduction}

\noindent The Leray--Schauder degree was introduced in \cite{LS} to get some rather pioneering existence results on Nonlinear Partial Differential Equations. It refines, very substantially, the finite-dimensional degree introduced by Brouwer \cite{Br} to prove his celebrated fixed point theorem.
The Leray--Schauder degree is a generalized topological counter of the number of zeros that a continuous map, $f$, can have on an open bounded subset, $\Omega$, of a real Banach space, $X$. To be  defined,  $f$ must be a compact perturbation of the identity map. Although this always occurs in finite-dimensional settings, it fails to be true in many important applications where the involved operators are not compact perturbations of the identity map but Fredholm operators of index zero between two Banach spaces $X$ and $Y$.  For Fredholm maps it is available the degree of Fredholm maps of Fitzpatrick,  Pejsachowicz and  Rabier \cite{FPR}, a refinement of the Elworthy and Tromba degree, based on the topological concepts of parity and orientation discussed by Fitzpatrick and Pejsachowicz in \cite{FP2}. Very recently, the authors of this article established in \cite{LSA} the hidden relationships between the degree for Fredholm maps of \cite{FPR} and the concept of generalized algebraic multiplicity of Esquinas and L\'opez-G\'omez in \cite{ELG,Es,LG01}, in a similar manner as the Schauder formula relates the Leray--Schauder degree to  the classic algebraic multiplicity. The main goal of this paper is axiomatizing the Fitzpatrick--Pejsachowicz--Rabier degree in the same vain as
the  Brouwer and Leray--Schauder degrees were axiomatized by F\"uhrer \cite{Fu} and Amann and Weiss \cite{AW}, respectively. In other words, we will give a  minimal set of properties that characterize the topological degree of Fitzpatrick, Pejsachowicz and Rabier.
\par
Throughout this paper, for any given pair of real Banach spaces $X, Y$ with $X\subset Y$, we denote by $\mathcal{L}_c(X,Y)$ the set of linear and continuous operators, $L \in \mathcal{L}(X,Y)$, that are a compact perturbation of the identity map, $L=I_X-K$. Then, the \emph{linear group}, $GL(X,Y)$ is defined
as the set of linear isomorphisms $L\in \mathcal{L}(X,Y)$. Similarly, the \emph{compact linear group},  $GL_c(X,Y)$, is defined  as $GL(X,Y)\cap \mathcal{L}_c(X,Y)$. For  any $L \in \mathcal{L}(X,Y)$, the sets $N[L]$ and $R[L]$ stand for the null space (kernel) and the range (image) of $L$, respectively.
An operator $L \in \mathcal{L}(X,Y)$ is said to be a  Fredholm operator if
$$
  \mathrm{dim\,}N[L]<\infty \quad \hbox{and}\quad \mathrm{codim\,}R[T]<\infty.
$$
In such case, $R[L]$ must be closed and the index of $L$ is defined by
$$
  \mathrm{ind\,} L := \mathrm{dim\,}N[L]-\mathrm{codim\,}R[L].
$$
In this paper, the set of Fredholm operators of index zero,  $L \in \mathcal{L}(X,Y)$, is denoted by $\Phi_0(X,Y)$. We also set $\Phi_0(X):=\Phi_0(X,X)$.
\par
In the context of the Leray--Schauder degree, for any pair of real Banach spaces $X, Y$ such that $X\subset Y$, any
open and bounded domain $\Omega\subset X$ and any map $f:\overline{\Omega}\subset X \to  Y$, it is said that $(f,\Omega)$ is an \emph{admissible pair} if:
	\begin{enumerate}
		\item[i)] $f\in \mathcal{C}(\overline{\Omega},Y)$;
		\item[ii)] $f$ is a compact perturbation of the identity map $I_X$;
		\item[iii)] $0\notin f(\partial\Omega)$.
	\end{enumerate}
	The class of admissible pairs will be denoted by $\mathscr{A}_{LS}$. Note that $(I_X,\Omega)\in \mathscr{A}_{LS}$ for every open and bounded subset $\Omega\subset X$ such that $0\notin\partial\Omega$. Actually, $(I_X,\Omega)\in \mathscr{A}_{LS,GL_c}$, where $\mathscr{A}_{LS,GL_c}$ stands for the
set of admissible pairs $(L,\Omega)\in \mathscr{A}_{LS}$ such that $L\in GL_c(X,Y)$. The next fundamental theorem establishes the existence and the uniqueness of the Leray--Schauder degree. The existence goes back to Leray and Schauder \cite{LS} and the uniqueness is attributable to Amann and Weiss \cite{AW},
though  F\"uhrer \cite{Fu} had already proven the uniqueness of the Brower degree when \cite{AW} was
published.
	
\begin{theorem}
\label{th1.1}
For any given pair of real Banach spaces, $X, Y$ such that $X\subset Y$, there exists a unique integer valued map, $\deg_{LS}:\mathscr{A}_{LS} \to \mathbb{Z}$, satisfying the following properties:
\begin{enumerate}
\item[{\rm (N)}]  \textbf{Normalization:} $\deg_{LS}(I_X,\Omega)=1$ if $0\in\Omega$.
			
\item[{\rm (A)}]  \textbf{Additivity:} For every $(f,\Omega)\in\mathscr{A}_{LS}$ and any pair of open disjoint subsets, $\Omega_{1}$ and $\Omega_{2}$, of $\Omega$ such that $0\notin f(\overline{\Omega}\backslash(\Omega_{1}\uplus \Omega_{2}))$,
\begin{equation}
\label{1.1}
			\deg_{LS}(f,\Omega)=\deg_{LS}(f,\Omega_{1})+\deg_{LS}(f,\Omega_{2}).
\end{equation}
			
\item[{\rm (H)}] \textbf{Homotopy Invariance:} For every homotopy $H\in \mathcal{C}([0,1]\times \overline{\Omega}, X)$ such that  $(H(t,\cdot),\Omega)\in\mathscr{A}_{LS}$ for each $t\in [0,1]$,
\begin{equation*}
			\deg_{LS}(H(0,\cdot),\Omega)=\deg_{LS}(H(1,\cdot),\Omega).
\end{equation*}
\end{enumerate}
Moreover, for every $(L,\Omega)\in\mathscr{A}_{LS,GL_c}$ with $0\in \Omega$,
\begin{equation}
\label{1.2}
		\deg_{LS}(L,\Omega)=(-1)^{\sum_{i=1}^{q}\mathfrak{m}_\mathrm{alg}[I_{X}-L;\mu_{i}]}
\end{equation}
where
\begin{equation*}
		\Spec(I_{X}-L)\cap (1,\infty)=\{\mu_{1},\mu_{2},...,\mu_{q}\}\,\qquad \mu_i\neq \mu_j\quad\hbox{if}\;\; i\neq j.
\end{equation*}
\end{theorem}
	
The map $\deg_{LS}$  is refereed to as the \emph{Leray--Schauder degree}. In \eqref{1.2},
setting $K:= I_X-L$, for any eigenvalue $\mu\in \Spec(K)$, we have denoted by $\mathfrak{m}_\mathrm{alg}[K;\mu]$ the classical algebraic multiplicity of $\mu$, i.e.,
\[
	\mathfrak{m}_{\mathrm{alg}}[K;\mu]= \mathrm{dim\,}\mathrm{Ker}[(\mu I_{X}-K)^{\nu(\mu)}],
\]
where $\nu(\mu)$ is the \emph{algebraic ascent} of $\mu$, i.e. the minimal integer, $\nu\geq 1$, such that
\[
	\mathrm{Ker}[(\mu-K)^{\nu}]=\mathrm{Ker}[(\mu I_{X}-K)^{\nu+1}].
\]
In Theorem \ref{th1.1}, the axiom (N) is called the \emph{normalization property} because, for every $n\in\mathbb{Z}$, the map $n\, \mathrm{deg}_{LS}$ also satisfies the axioms (A) and (H), though not (N). Thus, the axiom (N) normalizes the degree so that, for the identity map, it provides us with its  exact number of zeroes.  The axiom (A) packages three basic properties of the Leray--Schauder degree.  Indeed, by choosing $\Omega=\Omega_1=\Omega_2=\emptyset$, it becomes apparent that
\begin{equation}
	\label{1.3}
	\mathrm{deg}_{LS}(f,\emptyset)=0,
\end{equation}
so establishing that no continuous map can admit a zero in the empty set. Moreover, in the special case when $\Omega = \Omega_{1}\uplus \Omega_{2}$, \eqref{1.2} establishes the \emph{additivity property} of
the degree. Finally, in the special case when $\Omega_2=\emptyset$, it follows from \eqref{1.2}
and \eqref{1.3} that
\[
	\mathrm{deg}_{LS}(f,\Omega)= \mathrm{deg}_{LS}(f,\Omega_1),
\]
which is usually refereed to as the \emph{excision property} of the degree. If, in addition,
also $\Omega_1=\emptyset$, then
\[
	\mathrm{deg}_{LS}(f,\Omega)= 0 \quad \hbox{if}\;\; f^{-1}(0)\cap \overline{\Omega} =\emptyset.
\]
Therefore, for every $(f,\Omega)\in\mathscr{A}_{LS}$ such that $\deg_{LS}(f,\Omega)\neq 0$, the
equation $f(x)=0$ admits, at least, a solution in $\Omega$. This key property is
refereed to as the \emph{fundamental, or solution, property} of the degree.
\par
The axiom (H) establishes the \emph{invariance by homotopy} of the  degree.
It allows to calculate the degree in the practical situations of interest from the point of view of the applications. Not surprisingly, when dealing with analytic maps, $f$,  in $\mathbb{C}$, it provides us with the exact number of zeroes of $f$, counting orders, in $\Omega$ (see, e.g., Chapter 11 of \cite{LGCV}).
\par
From a geometrical point of view, the construction of the Leray--Schauder degree relies on the concept of orientation for $X=Y$. Let $H\in \mathcal{C}([0,1]\times\overline{\Omega},X)$ be a homotopy with $(H(t,\cdot),\Omega)\in\mathscr{A}_{LS,GL_c}$ for each $t\in [0,1]$. Since $H$ can be regarded as the continuous path  $\mathfrak{L}\in \mathcal{C}([0,1], GL_{c}(X))$ defined by
$\mathfrak{L}(t):=H(\cdot,t)$, $t\in [0,1]$, by the axiom (H), the integer $\deg_{LS}(\mathfrak{L}(t),\Omega)$ is constant for all $t\in[0,1]$. This introduces an equivalence relation between the operators of $GL_c(X)$. Indeed, for every pair of operators $L_0$, $L_1\in GL_c(X)$, it is said that $L_0 \sim L_1$ if $L_0$ and $L_1$ are homotopic in $\mathscr{A}_{LS,GL_c}$ in the sense that $L_0=\mathfrak{L}(0)$ and $L_1= \mathfrak{L}(1)$ for some curve
$\mathfrak{L}\in \mathcal{C}([0,1], GL_{c}(X))$. This equivalence relation divides $GL_{c}(X)$ into two path connected components, $GL^{+}_{c}(X)$ and $GL^{-}_{c}(X)$, separated away by $\mathcal{S}(X)\cap GL_c(X)$, where
\[
\mathcal{S}(X):=\mathcal{L}(X)\setminus GL(X).
\]
Conversely, if $GL^{+}_{c}(X)$  stands for the path connected component of $GL_c(X)$ containing $I_X$, the fact that a given operator $L\in GL_c(X)$ belongs to one component, or another, defines an \emph{orientation} on $L$. This allows us to define a map,
\begin{equation}
\label{b.4}
\deg_{LS}(L,\Omega):=\left\{\begin{array}{ll}
1 & \text{ if } L\in GL^{+}_{c}(X)\;\;\hbox{and}\;\; 0\in \Omega, \\
-1 & \text{ if } L\in GL^{-}_{c}(X)\;\;\hbox{and}\;\; 0\in \Omega,\\
0 & \text{ if } L\in GL_{c}(X)\;\;\hbox{and}\;\; 0\notin \Omega,
\end{array}\right.
\end{equation}
verifying the three axioms of the Leray--Schauder degree in the class $\mathscr{A}_{LS,GL_c}$ and, in particular, the homotopy invariance. Once defined the degree in $\mathscr{A}_{LS,GL_c}$, one can extend this restricted concept of degree to the regular pairs $(f,\Omega)\in\mathscr{R}_{LS}$ through the identity
$$
\deg_{LS}(f,\Omega)=\sum_{x\in f^{-1}(0)\cap\Omega}\deg_{LS}(Df(x),\Omega).
$$
Finally, according to the Sard--Smale theorem and the homotopy invariance property, it can be extended to be defined for general admisible pairs, $(f,\Omega)\in\mathscr{A}_{LS}$. A crucial feature that
facilitates this construction of the degree is the fact that the space $GL_{c}(X)$ consists of two path-connected components. Thus, it admits an orientation. This fails to be true in the general context of Fredholm operators of index zero, which makes the mathematical analysis of this paper much more sophisticated technically.
\par
The main goal of this paper is establishing an analogous of Theorem \ref{th1.1} for Fredholm Operators of index zero within the context of the degree for Fredholm maps of Fitzpatrick,  Pejsachowicz and  Rabier \cite{FPR}. Let $\Omega$ be an open and bounded subset of a real Banach space $X$. Then, an operator $f:\overline{\Omega}\subset X \to Y$ is said to be $\mathcal{C}^{1}$-\emph{Fredholm of index zero} if
\begin{equation*}
f\in \mathcal{C}^{1}(\overline{\Omega},Y)\quad \hbox{and}\quad
Df\in \mathcal{C}(\overline{\Omega},\Phi_{0}(X,Y)).
\end{equation*}
In  this paper, the set of all these operators is denoted by $\mathscr{F}^{1}_{0}(\Omega,Y)$.
A given operator $f\in\mathscr{F}^{1}_{0}(\Omega,Y)$ is said to be \emph{orientable} when the image  $Df(\Omega)$ is an orientable subset of $\Phi_{0}(X,Y)$ (see Section 3 for the concept of orientability).  Moreover, for any open and bounded subset, $\Omega$, of $X$ and any map $f:\overline{\Omega}\subset X \to Y$ satisfying 	
\begin{enumerate}
	\item $f\in \mathscr{F}^{1}_{0}(\Omega,Y)$ is \emph{orientable} with orientation $\varepsilon$,
	\item $f$ is \emph{proper} in $\overline{\Omega}$, i.e., $f^{-1}(K)$ is compact for every compact subset $K\subset Y$,
	\item $0\notin f(\partial \Omega)$,
\end{enumerate}
it will be said that $(f,\Omega,\varepsilon)$ is a \emph{Fredholm admissible tern}. The set of all
Fredholm admissible terns in the context of Fitzpatrick,  Pejsachowicz and  Rabier \cite{FPR}
is denoted by $\mathscr{A}$. Given $(f,\Omega,\varepsilon)\in \mathscr{A}$,
it is said that $(f,\Omega,\varepsilon)$ is a regular tern if $0$ is a regular value of $f$, i.e.,
$Df(x)\in GL(X,Y)$ for all $x \in f^{-1}(0)$. The set of regular terns is denoted by $\mathscr{R}$.
Lastly, a map $H\in \mathcal{C}^{1}([0,1]\times \overline{\Omega},Y)$ is said to be a    $\mathcal{C}^{1}$-\emph{Fredholm homotopy} if $D_{x}H(t,\cdot)\in\Phi_{0}(X,Y)$ for each $t\in[0,1]$, and it is called \textit{orientable} if $D_{x}H([0,1]\times \Omega)$ is an orientable subset of $\Phi_{0}(X,Y)$. The main theorem of this paper reads  as follows.

\begin{theorem}
\label{th1.2}
There exists a unique integer valued map
$$
   \deg: \mathscr{A} \to \mathbb{Z}
$$
satisfying the next properties:
\begin{enumerate}
\item[{\rm (N)}] \textbf{Normalization:} $\deg(L,\Omega,\varepsilon)=\varepsilon(L)$ for all
		$L\in GL(X,Y)$ if $0\in \Omega$.
		
\item[{\rm (A)}] \textbf{Additivity:} For every $(f,\Omega,\varepsilon)\in\mathscr{A}$  and any
		pair of disjoint open subsets $\Omega_{1}$ and $\Omega_{2}$ of $\Omega$ with $0\notin f(\Omega\backslash (\Omega_{1}\uplus \Omega_{2}))$,
\begin{equation*}
 \deg(f,\Omega,\varepsilon)=\deg(f,\Omega_{1},\varepsilon)+\deg(f,\Omega_{2},\varepsilon).
\end{equation*}
		
\item[{\rm (H)}] \textbf{Homotopy Invariance:} For each proper $\mathcal{C}^{1}$-Fredholm homotopy
		$H\in \mathcal{C}^{1}([0,1]\times \overline{\Omega}, Y)$ with orientation $\varepsilon$ such that $(H(t,\cdot),\Omega,\varepsilon_{t})\in\mathscr{A}$ for each $t\in[0,1]$
\begin{equation*}
		\deg(H(0,\cdot),\Omega,\varepsilon_{0})=\deg(H(1,\cdot),\Omega,\varepsilon_{1}).
		\end{equation*}
\end{enumerate}

Moreover, for every $(f,\Omega,\varepsilon)\in\mathscr{R}$ satisfying $Df(\Omega)\cap GL(X,Y)\neq\emptyset$ and $L\in Df(\Omega)\cap GL(X,Y)$
\begin{equation}
\label{1.4}
\deg(f,\Omega,\varepsilon)=\varepsilon(L)\cdot \sum_{x\in f^{-1}(0)\cap \Omega} (-1)^{\sum_{\lambda_{x}\in\Sigma(\mathfrak{L}_{\omega,x})}\chi[\mathfrak{L}_{\omega,x};\lambda_{x}]}
\end{equation}
where $\mathfrak{L}_{\omega,x}\in \mathscr{C}^{\omega}([a,b],\Phi_{0}(X,Y))$ is an analytical path $\mathcal{A}$-homotopic to some path $\mathfrak{L}_{x}\in\mathcal{C}([a,b],Df(\Omega))$ joining $Df(x)$ and $L$ (see Section 3 for the precise meaning), and $\chi$ is the generalized algebraic multiplicity introduced by Esquinas and L\'{o}pez-G\'{o}mez in \cite{ELG,Es,LG01}
(see Section 2 for its definition and main properties).
\end{theorem}

As in the context of the Leray--Schauder degree, the axiom (A) packages three fundamental properties of the degree. Namely, the additivity and excision properties, as well as the existence property, that is, whenever $(f,\Omega,\varepsilon)\in\mathscr{A}$ satisfies $\deg(f,\Omega,\varepsilon)\neq 0$, there exists $x\in \Omega$ such that $f(x)=0$.
\par
The existence of the map $\deg$ was established by Fitzpatrick, Pejsachowicz and Rabier in \cite{FPR} based on the concept of orientability introduced by Fitzpatrick and Pejsachowicz \cite{FP2}. The identity \eqref{1.4} is a substantial sharpening of the classical Leray--Schauder formula in the context of the degree for Fredholm maps; it was  proven by the authors in \cite{LSA}.  Thus, the main novelty of Theorem \ref{th1.2} is establishing the uniqueness of $\deg$ as a direct consequence of \eqref{1.4}; so, establishing an analogue of Theorem \ref{th1.1}. Obviously, in the context of Leray-Schauder degree, if we restrict the admissible pairs $(f,\Omega)$ to those
 of the  form $f:\overline{\Omega}\subset X\to Y$ with $f$ of class $\mathcal{C}^{1}$  (in the sequel we will denote this class by $\mathscr{A}_{LS}^{1}$), then Theorem \ref{th1.2} generalizes, very substantially, 	Theorem \ref{th1.1}.
\par
It is convenient to remark that the orientability introduced by Fitzpatrick and Pejsachowicz in \cite{FP2} deals with continuous maps $h:\Lambda\to\Phi_{0}(X,Y)$ where $\Lambda$ is a given topological space. As in this article we will adopt a more geometrical perspective, we are orienting subspaces $\Omega\subset \Phi_{0}(X,Y)$, instead of maps. However, given a continuous map $h:\Lambda\to\Phi_{0}(X,Y)$, if $h(\Lambda)\subset \Phi_{0}(X,Y)$ is orientable in our sense, necessarily $h:\Lambda\to\Phi_{0}(X,Y)$ is orientable as discussed  in \cite{FP2}. Therefore, our concept of  orientability implies the original one of \cite{FPR}. Thus, all the considerations therein apply in our setting.
\par
Benevieri and Furi \cite{BF3} have established the uniqueness of another formulation of the topological degree for Fredholm operators \cite{BF1,BF2}. In particular, by using different techniques, they introduced another concept of orientability for continuous maps $h:\Lambda\to \Phi_{0}(X,Y)$ on a  topological space $\Lambda$. When  $h$ has a regular point, i.e.,
$$
R_{h}:=\{p\in\Lambda: h(p)\in GL(X,Y)\}\neq \emptyset,
$$	
the two notions coincide in the sense that $h:\Lambda\to \Phi_{0}(X,Y)$ is orientable in the Benivieri-Furi sense (BF-orientable for short) if and only if it is Fitzpatrick-Pejsachowicz orientable (FP-orientable for short). However, when $R_{h}=\emptyset$, these two concepts are different. Although the singular maps with $R_{h}=\emptyset$ are orientable adopting the FP-orientation, there are examples of singular $h$'s that are BF-orientable, while others are not (see \cite[Section 5]{BF2} for further details). Due to this fact, the degree constructed by Benevieri and Furi does not coincide with the degree of Fitzpatrick, Pejsachowicz and Rabier, because there are admissible terns $(f,\Omega,\varepsilon)$ such that $Df:\overline{\Omega}\to \Phi_{0}(X,Y)$ is not BF-orientable. Thus, although Benevieri and Furi proved in \cite{BF3} an uniqueness
result for their degree, our Theorem \ref{th1.2} here is independent of their main uniqueness result. 
Actually, both uniqueness results are independent in the sense that no one implies the other, though 
in some important applications both degrees coincide.  However, since the algebraic multiplicity $\chi$ is defined for Fredholm operator curves $\mathfrak{L}:[a,b]\to\Phi_{0}(X,Y)$ and the orientability notion of Fitzpatrick and Pejsachowicz is defined through the use of this type of curves by means of  their notion of parity, we see far more natural the degree of Fitzpatrick, Pejsachowicz and Rabier for delivering an analogue of the uniqueness theorem of Amann and Weiss through \eqref{1.4}, within the same vain as in the classical context of the Leray--Schauder degree.
\par
The distribution of this paper is the following. Sections 2 contains some necessary preliminaries
on the Leray--Schauder degree and the generalized algebraic multiplicity, $\chi$, used in the
generalized Leray--Schauder formula \eqref{1.4}. Section 3 introduces the concepts of parity and
orientation of Fitzpatrick and  Pejsachowicz \cite{FP2} and collects some of the findings of the
authors in \cite{LSA}, where the Fitzpatrick--Pejsachowicz parity, $\sigma$,  was calculated through the generalized algebraic multiplicity $\chi$. These results are needed for axiomatizing the parity $\sigma$
in Section 4. The main result of Section 4 is Theorem \ref{th4.2}, which characterizes $\sigma$ through
a normalization property, a product formula and its invariance by homotopy, by means of the
algebraic multiplicity $\chi$. This result is reminiscent of the uniqueness theorem of Mora-Corral \cite{MC} for the multiplicity $\chi$ (see also Chapter 6 of \cite{LGMC}). Finally, based on these results, the proof of Theorem \ref{th1.2} is delivered in Section 5 after revisiting, very shortly, the main concepts of the  Fitzpatrick--Pejsachowicz--Rabier degree.

\section{Generalized Algebraic Multiplicity}

\noindent As the generalized algebraic multiplicity introduced by Esquinas and L\'{o}pez-G\'{o}mez in \cite{ELG,Es,LG01} is a pivotal technical device in the proof of Theorem \ref{th1.2} through the formula \eqref{1.4}, we will collect some of its most  fundamental properties, among them the uniqueness theorem of Mora-Corral \cite{MC,LGMC}.
\par
Given two Banach spaces, $X$ and $Y$, by a  \emph{Fredholm path, or curve,} it is meant any map $\mathfrak{L}\in \mathcal{C}([a,b],\Phi_{0}(X,Y))$.
Given a Fredholm path, $\mathfrak{L}$, it is said that $\lambda\in[a,b]$ is an \emph{eigenvalue} of $\mathfrak{L}$ if $\mathfrak{L}(\lambda)\notin GL(X,Y)$. Then, the \emph{spectrum} of $\mathfrak{L}$, $\Sigma(\mathfrak{L})$,  consists of the set of all these eigenvalues, i.e.,  	
\begin{equation*}
\Sigma(\mathfrak{L}):=\{\lambda\in[a,b]: \mathfrak{L}(\lambda)\notin GL(X,Y)\}.
\end{equation*}
According to Lemma 6.1.1 of \cite{LG01}, $\Sigma(\mathfrak{L})$ is a compact subset of $[a,b]$, though,
in general, one cannot say anything more about it, because for any given compact subset of $[a,b]$, $J$,
there exists a continuous function $\mathfrak{L} :[a,b]\to \mathbb{R}$ such that $J=\mathfrak{L}^{-1}(0)$. Next, we will deliver a concept  introduced in \cite{LG01} to characterize whether, or not, the algebraic multiplicity of  Esquinas and L\'{o}pez-G\'{o}mez \cite{ELG,Es,LG01} is well defined. Let $\mathfrak{L}\in \mathcal{C}([a,b], \Phi_{0}(X,Y))$ and $k\in\mathbb{N}$. An eigenvalue  $\lambda_{0}\in\Sigma(\mathfrak{L})$ is said to be a  \emph{$k$-algebraic eigenvalue} if there exits $\varepsilon>0$ such that
\begin{enumerate}
	\item[{\rm (a)}] $\mathfrak{L}(\lambda)\in GL(X,Y)$ if $0<|\lambda-\lambda_0|<\varepsilon$;
	\item[{\rm (b)}] There exits $C>0$ such that
	\begin{equation}
	\label{b.5}
	\|\mathfrak{L}^{-1}(\lambda)\|<\frac{C}{|\lambda-\lambda_{0}|^{k}}\quad\hbox{if}\;\;
	0<|\lambda-\lambda_0|<\varepsilon;
	\end{equation}
	\item[{\rm (c)}] $k$ is the least positive integer for which \eqref{b.5} holds.
\end{enumerate}
The set of algebraic eigenvalues of $\mathfrak{L}$ or order $k$ will be denoted by $\Alg_k(\mathfrak{L})$. Thus, the set of \emph{algebraic eigenvalues} can be defined by
\[
\Alg(\mathfrak{L}):=\biguplus_{k\in\mathbb{N}}\Alg_k(\mathfrak{L}).
\]
According to Theorems 4.4.1 and 4.4.4 of \cite{LG01}, when $\mathfrak{L}(\lambda)$ is real analytic in $[a,b]$, i.e., $\mathfrak{L}\in\mathcal{C}^{\omega}([a,b], \Phi_{0}(X,Y))$, then either $\Sigma(\mathfrak{L})=[a,b]$, or $\Sigma(\mathfrak{L})$ is finite and $\Sigma(\mathfrak{L})\subset \Alg(\mathfrak{L})$.
\par
According to \cite[Ch. 7]{LGMC}, $\lambda_0\in \Alg(\mathfrak{L})$ if, and only if, the lengths of all Jordan chains of $\mathfrak{L}$ at 	$\lambda_0$ are uniformly bounded above, which allows to characterize whether, or not, $\mathfrak{L}(\lambda)$ admits a local Smith form at $\lambda_0$ (see \cite{LGMC}). The next concept allows to introduce a generalized algebraic multiplicity, $\chi[\mathfrak{L},\lambda_0]$, in a rather natural manner.
It goes back to \cite{ELG}. Subsequently, we will denote
$$
\mathfrak{L}_{j}:=\frac{1}{j!}\mathfrak{L}^{(j)}(\lambda_{0}), \quad 1\leq j\leq r,
$$
if these derivatives exist. Given a path $\mathfrak{L}\in \mathcal{C}^{r}([a,b],\Phi_{0}(X,Y))$ and an integer  $1\leq k \leq r$, a given eigenvalue $\lambda_{0}\in \Sigma(\mathfrak{L})$ is said to be a \emph{$k$-transversal eigenvalue} of $\mathfrak{L}$ if
\begin{equation*}
\bigoplus_{j=1}^{k}\mathfrak{L}_{j}\left(\bigcap_{i=0}^{j-1}\Ker(\mathfrak{L}_{i})\right)
\oplus R(\mathfrak{L}_{0})=Y\;\; \hbox{with}\;\; \mathfrak{L}_{k}\left(\bigcap_{i=0}^{k-1}\Ker(\mathfrak{L}_{i})\right)\neq \{0\}.
\end{equation*}
For these eigenvalues, the \emph{algebraic multiplicity of $\mathfrak{L}$ at $\lambda_{0}$}, $\chi[\mathfrak{L},\lambda_0]$,  is defined through
\begin{equation}
\label{b.6}
\chi[\mathfrak{L}; \lambda_{0}] :=\sum_{j=1}^{k}j\cdot \dim \mathfrak{L}_{j}\left(\bigcap_{i=0}^{j-1}\Ker(\mathfrak{L}_{i})\right).
\end{equation}
By Theorems 4.3.2 and 5.3.3 of \cite{LG01}, for every $\mathfrak{L}\in \mathcal{C}^{r}([a,b], \Phi_{0}(X,Y))$, $k\in\{1,2,...,r\}$ and $\lambda_{0}\in \Alg_{k}(\mathfrak{L})$, there exists a polynomial $\Phi: \mathbb{R}\to \mathcal{L}(X)$ with $\Phi(\lambda_{0})=I_{X}$ such that $\lambda_{0}$ is a $k$-transversal eigenvalue of the path
\begin{equation}
\label{b.7}
\mathfrak{L}^{\Phi}:=\mathfrak{L}\circ\Phi\in \mathcal{C}^{r}([a,b], \Phi_{0}(X,Y)).
\end{equation}
Moreover, $\chi[\mathfrak{L}^{\Phi};\lambda_{0}]$ is independent of the curve of \emph{trasversalizing local isomorphisms} $\Phi$ chosen to transversalize $\mathfrak{L}$ at $\lambda_0$ through \eqref{b.7}, regardless $\Phi$ is a polynomial or not. Therefore, the next generalized concept of algebraic multiplicity is consistent
\[
\chi[\mathfrak{L};\lambda_0]:= \chi[\mathfrak{L}^\Phi;\lambda_0].
\]
This concept of algebraic multiplicity can be easily extended by setting
\[
\chi[\mathfrak{L};\lambda_0] =0 \quad \hbox{if}\;\; \lambda_0\notin\Sigma(\mathfrak{L})
\]
and
\[
\chi[\mathfrak{L};\lambda_0] =+\infty \quad \hbox{if}\;\; \lambda_0\in \Sigma(\mathfrak{L})
\setminus \Alg(\mathfrak{L}) \;\; \hbox{and}\;\; r=+\infty.
\]
Thus, $\chi[\mathfrak{L};\lambda]$ is well defined for all  $\lambda\in [a,b]$ of any smooth path $\mathfrak{L}\in \mathcal{C}^{\infty}([a,b],\Phi_{0}(X,Y))$ and, in particular, for any analytical curve  $\mathfrak{L}\in\mathcal{C}^{\omega}([a,b],\Phi_{0}(X,Y))$. The next uniqueness result goes back to
Mora-Corral \cite{MC} and \cite[Ch. 6]{LGMC}.

\begin{theorem}
	\label{th2.1}
	For every $\varepsilon>0$, the algebraic multiplicity $\chi$ is the unique map 	
	\begin{equation*}
	\chi[\cdot; \lambda_{0}]: \mathcal{C}^{\infty}((\lambda_{0}-\varepsilon,\lambda_{0}+\varepsilon), \Phi_{0}(X))\longrightarrow [0,\infty]
	\end{equation*}
	satisfying the next two axioms:
	\begin{enumerate}
		\item[{\rm (P)}] For every pair $\mathfrak{L}, \mathfrak{M} \in \mathcal{C}^{\infty}((\lambda_{0}-\varepsilon,\lambda_{0}+\varepsilon), \Phi_{0}(X))$,
		\begin{equation*}
		\chi[\mathfrak{L}\circ\mathfrak{M};\lambda_{0}]
		=\chi[\mathfrak{L};\lambda_{0}]+\chi[\mathfrak{M};\lambda_{0}].
		\end{equation*}
		\item[{\rm (N)}] There exits a rank one projection $P_{0}\in L(X)$ such that
		\begin{equation*}
		\chi[(\lambda-\lambda_{0})P_{0}+I_{X}-P_{0};\lambda_{0}]=1.
		\end{equation*}
	\end{enumerate}
\end{theorem}

The axiom (P) is the  \emph{product formula} and the axiom (N) is a \emph{normalization property}
for establishing the uniqueness of the algebraic multiplicity. From these axioms one can derive all the remaining properties of the generalized algebraic multiplicity $\chi$. Among them, that it equals the classical algebraic multiplicity when
\[
\mathfrak{L}(\lambda)= \lambda I_{X} - K
\]
for some compact operator $K$. Indeed, according to \cite{LGMC}, for every smoth path $\mathfrak{L}\in \mathcal{C}^{\infty}((\lambda_{0}-\varepsilon,\lambda_{0}+\varepsilon),\Phi_{0}(X))$, the following properties hold:
\begin{itemize}
	\item $\chi[\mathfrak{L};\lambda_{0}]\in\mathbb{N}\uplus\{+\infty\}$;
	\item $\chi[\mathfrak{L};\lambda_{0}]=0$ if, and only if, $\mathfrak{L}(\lambda_0)
	\in GL(X)$;
	\item $\chi[\mathfrak{L};\lambda_{0}]<\infty$ if, and only if, $\lambda_0 \in\Alg(\mathfrak{L})$.
	\item If $X =\mathbb{R}^N$, then, in any basis,
	\begin{equation*}
	\chi[\mathfrak{L};\lambda_{0}]= \mathrm{ord}_{\lambda_{0}}\det \mathfrak{L}(\lambda).
	\end{equation*}
	\item Let $L\in \mathcal{L}(X)$ be such that $\lambda I_X-L\in \Phi_0(X)$. Then, for every $\lambda_0\in \Spec(L)$, there exists $k\geq 1$ such that
	\begin{equation}
	\label{b.8}
	\begin{split}
	\chi [\lambda I_X-L;\lambda_{0}] & =\underset{n\in\mathbb{N}}{\sup} \dim \Ker [(\lambda_{0}I_{X}-L)^{n}] \\ & = \dim \Ker [(\lambda_{0}I_{X}-L)^{k}]=\mathfrak{m}_\mathrm{alg}[L;\lambda_0].
	\end{split}
	\end{equation}
\end{itemize}
Therefore, $\chi$ extends, very substantially, the classical concept of algebraic multiplicity.
	
\section{Parity and Orientability}
\noindent This section collects some very recent findings of the authors in \cite{LSA}
in connection with the concepts of \emph{parity} and \emph{orientability} introduced by  Fitzpatrick and
Pejsachowicz in \cite{FP2}. We begin by recalling some important features concerning the structure of the  space of linear Fredholm operators of index zero, $\Phi_{0}(X,Y)$, which is an open path connected subset of $\mathcal{L}(X,Y)$;  in general, $\Phi_{0}(X,Y)$ is not linear. Subsequently, for every $n\in\mathbb{N}$, we denote by $\mathcal{S}_{n}(X,Y)$ the set of \textit{singular operators of order} $n$
\begin{equation*}
	\mathcal{S}_{n}(X,Y):=\{L\in\Phi_{0}(X,Y):\;\; \dim N[L] =n\}.
\end{equation*}
Thus, the set of \textit{singular operators} is given through
\begin{equation*}
	\mathcal{S}(X,Y):=\Phi_{0}(X,Y)\backslash GL(X,Y)=\biguplus_{n\in\mathbb{N}}\mathcal{S}_{n}(X,Y).
\end{equation*}
According to \cite{FP1}, for every $n\in\mathbb{N}$, $\mathcal{S}_{n}(X,Y)$ is a Banach submanifold of $\Phi_{0}(X,Y)$ of codimension $n^{2}$. This feature allows us to view  $\mathcal{S}(X,Y)$ as an
hypersurface of $\Phi_{0}(X,Y)$. By Theorem 1 of Kuiper \cite{K}, the space of isomorphisms,
$GL({H})$, of any separable infinite dimensional Hilbert space, ${H}$,  is path connected. Thus, it is
not possible to introduce an orientation in $GL(X,Y)$ for general Banach spaces $X, Y$, since in general, $GL(X,Y)$ is path connected. This fact reveals a fundamental difference between finite and infinite dimensional vector normed spaces, because, for every $N\in\mathbb{N}$, the space
$ GL(\mathbb{R}^{N})
$
is divided into two path connected components, $GL^\pm(\mathbb{R}^N)$.
\par
A key technical tool to overcome this difficulty in order to define a degree in $\Phi_0(X,Y)$  is provided by the concept of \emph{parity} introduced by  Fitzpatrick and Pejsachowicz \cite{FP2}.  The parity is a generalized local detector of the change of orientability of a given \emph{admissible path}. Subsequently, a Fredholm path $\mathfrak{L}\in \mathcal{C}([a,b],\Phi_{0}(X,Y))$ is said to be \emph{admissible} if $\mathfrak{L}(a), \mathfrak{L}(b)\in GL(X,Y)$, and we denote by $\mathscr{C}([a,b],\Phi_{0}(X,Y))$
the set of admissible paths. Moreover, for every $r\in\mathbb{N}\uplus\{+\infty,\omega\}$,
 we set
\[
	\mathscr{C}^r([a,b],\Phi_{0}(X,Y)) := \mathcal{C}^{r}([a,b],\Phi_{0}(X,Y))\cap \mathscr{C}([a,b],\Phi_{0}(X,Y)).
\]
The fastest way to introduce the notion of parity consists in defining it for $\mathscr{C}$-transversal  paths and then for general admissible curves  through the density of $\mathscr{C}$-transversal paths in $\mathscr{C}([a,b],\Phi_{0}(X,Y))$, already established in \cite{FP1}. A  Fredholm path, $\mathfrak{L}\in \mathcal{C}([a,b],\Phi_{0}(X,Y))$, is said to be  \emph{$\mathscr{C}$-transversal} if
	\begin{enumerate}
		\item[{\rm i)}] $\mathfrak{L}\in \mathscr{C}^{1}([a,b],\Phi_{0}(X,Y))$;
		\item[{\rm ii)}] $\mathfrak{L}([a,b])\cap \mathcal{S}(X,Y)\subset \mathcal{S}_{1}(X,Y)$ and it is finite;
		\item[{\rm iii)}] $\mathfrak{L}$ is transversal to $\mathcal{S}_{1}(X,Y)$ at each point of $\mathfrak{L}([a,b])\cap \mathcal{S}(X,Y)$.
	\end{enumerate}
The path $\mathfrak{L}\in\mathcal{C}^{1}([a,b],\Phi_{0}(X,Y))$ is said to be transversal to $\mathcal{S}_{1}(X,Y)$ at $\lambda_{0}$ if
	\begin{equation*}
	\mathfrak{L}'(\lambda_{0})+T_{\mathfrak{L}(\lambda_{0})}\mathcal{S}_{1}(X,Y)=\mathcal{L}(X,Y),
	\end{equation*}
	where $T_{\mathfrak{L}(\lambda_{0})}\mathcal{S}_{1}(X,Y)$ stands for the tangent space to the manifold $\mathcal{S}_{1}(X,Y)$ at $\mathfrak{L}(\lambda_{0})$.
	\par
	When $\mathfrak{L}$ is $\mathscr{C}$-transversal, the \emph{parity} of $\mathfrak{L}$ in $[a,b]$ is defined by
	\begin{equation*}
	\sigma(\mathfrak{L},[a,b]):=(-1)^{k},
	\end{equation*}
	where $k\in\mathbb{N}$ equals the cardinal of $\mathfrak{L}([a,b])\cap \mathcal{S}(X,Y)$.
	Thus, the parity of a $\mathscr{C}$-transversal path, $\mathfrak{L}(\lambda)$, is the number of times, mod 2, that $\mathfrak{L}(\lambda)$ intersects transversally the singular hypersurface $\mathcal{S}(X,Y)$.
	\par
	The fact that the $\mathscr{C}$-transversal paths are dense in the set of all admissible paths, together with the next stability property: for any $\mathscr{C}$-transversal path $\mathfrak{L}\in \mathcal{C}([a,b],\Phi_{0}(X,Y))$, there exists $\varepsilon>0$ such that
\begin{equation*}
\sigma(\mathfrak{L},[a,b])=\sigma(\tilde{\mathfrak{L}},[a,b])
\end{equation*}
for all $\mathscr{C}$-transversal path  $\tilde {\mathfrak{L}}\in \mathcal{C}([a,b],\Phi_{0}(X,Y))$
with  $\|\mathfrak{L}-\tilde{\mathfrak{L}}\|_{\infty}<\varepsilon$ (see \cite{FP1}); allows us
to define the parity for a general admissible path $\mathfrak{L}\in \mathscr{C}([a,b],\Phi_{0}(X,Y))$ through
\begin{equation*}
	\sigma(\mathfrak{L},[a,b]):=\sigma(\tilde{\mathfrak{L}},[a,b]),
\end{equation*}
where $\tilde{\mathfrak{L}}$ is any $\mathscr{C}$-transversal curve satisfying $\|\mathfrak{L}-\tilde{\mathfrak{L}}\|_{\infty}<\varepsilon$ for sufficiently small $\varepsilon>0$.
\par
Subsequently, a given homotopy $H\in\mathcal{C}([0,1]\times[a,b],\Phi_{0}(X,Y))$  is said to be  \emph{admissible} if $H([0,1]\times\{a,b\})\subset GL(X,Y)$. Moreover, two given paths, $\mathfrak{L}_{1}$ and $\mathfrak{L}_{2}$, are said to be $\mathcal{A}$-\emph{homotopic} if they are homotopic through an admissible homotopy. A fundamental property of the parity is its invariance under admissible homotopies, which was established in \cite{FP2}.
\par
	The next result, proven by the authors in \cite{LSA}, establishes that, as soon as the Fredholm path $\mathfrak{L}(\lambda)$ is defined in
	$\mathcal{L}_{c}(X)$, every transversal intersection with $\mathcal{S}(X)$ induces a change of orientation, i.e., a change of path-connected component.
	
\begin{theorem}
		\label{th3.2}
		Let $\mathfrak{L}\in \mathscr{C}([a,b],\mathcal{L}_{c}(X))$ be an admisible curve with values in $\mathcal{L}_{c}(X)$. Then, $\sigma(\mathfrak{L},[a,b])=-1$ if, and only if, $\mathfrak{L}(a)$ and $\mathfrak{L}(b)$ lye in different path-connected components of $GL_{c}(X)$.
\end{theorem}

Theorem \ref{th3.2} motivates the geometrical interpretation of the parity as a local detector of the change of orientation of the operators of a Fredholm path. As illustrated by Figure \ref{Fig3}, each transversal intersection of the path $\mathfrak{L}(\lambda)$ with the singular hypersurface $\mathcal{S}(X)$ can be viewed as a change of path-connected component.
	
	\begin{figure}
		\begin{center}

			\tikzset{every picture/.style={line width=0.75pt}} 
			
			\begin{tikzpicture}[x=0.75pt,y=0.75pt,yscale=-1,xscale=1]
			
			\draw   (74,160.09) .. controls (74,126.9) and (100.27,100) .. (132.69,100) .. controls (165.1,100) and (191.37,126.9) .. (191.37,160.09) .. controls (191.37,193.28) and (165.1,220.19) .. (132.69,220.19) .. controls (100.27,220.19) and (74,193.28) .. (74,160.09) -- cycle ;
			\draw   (242.13,160.91) .. controls (242.13,127.72) and (268.4,100.81) .. (300.81,100.81) .. controls (333.23,100.81) and (359.5,127.72) .. (359.5,160.91) .. controls (359.5,194.1) and (333.23,221) .. (300.81,221) .. controls (268.4,221) and (242.13,194.1) .. (242.13,160.91) -- cycle ;
			\draw    (76.38,168.21) .. controls (99.38,190.14) and (158.06,192.58) .. (189.79,168.21) ;

			\draw    (243.71,166.59) .. controls (266.71,188.52) and (325.4,190.95) .. (357.12,166.59) ;

			\draw  [dash pattern={on 4.5pt off 4.5pt}]  (76.38,168.21) .. controls (108.1,143.85) and (163.62,143.04) .. (189.79,168.21) ;

			\draw  [dash pattern={on 4.5pt off 4.5pt}]  (243.71,166.59) .. controls (275.44,142.23) and (330.95,141.42) .. (357.12,166.59) ;

			\draw    (124.76,150.35) -- (112.07,182.83) ;

			\draw    (146.96,150.35) -- (134.27,186.08) ;

			\draw    (165.99,154.41) -- (156.48,182.02) ;

			\draw    (102.55,154.41) -- (92.24,177.15) ;

			\draw    (271.47,152.79) -- (261.16,175.52) ;

			\draw    (290.5,148.72) -- (277.82,181.21) ;

			\draw    (309.54,147.1) -- (296.85,182.83) ;

			\draw    (328.57,152.79) -- (319.05,180.4) ;

			\draw    (93.03,136.54) .. controls (124.76,112.18) and (115.64,229.93) .. (147.36,205.57) ;

			\draw    (270.28,121.93) .. controls (302,97.56) and (288.92,137.36) .. (279.4,162.53) ;

			\draw  [dash pattern={on 4.5pt off 4.5pt}]  (279.4,159.28) .. controls (275.69,169.74) and (276.49,164.87) .. (273.31,181.92) ;

			\draw    (273.31,181.92) .. controls (259.58,213.69) and (317.86,225.87) .. (304.38,184.46) ;

			\draw  [dash pattern={on 4.5pt off 4.5pt}]  (304.38,184.46) .. controls (296.72,171.46) and (311.39,170.65) .. (312.18,167.4) ;

			\draw    (312.18,167.4) .. controls (343.9,143.04) and (296.85,149.54) .. (319.05,124.36) ;

			\draw  [fill={rgb, 255:red, 0; green, 0; blue, 0 }  ,fill opacity=1 ] (89.86,139.79) .. controls (89.86,138) and (91.28,136.54) .. (93.03,136.54) .. controls (94.79,136.54) and (96.21,138) .. (96.21,139.79) .. controls (96.21,141.59) and (94.79,143.04) .. (93.03,143.04) .. controls (91.28,143.04) and (89.86,141.59) .. (89.86,139.79) -- cycle ;
			\draw  [fill={rgb, 255:red, 0; green, 0; blue, 0 }  ,fill opacity=1 ] (312.18,167.4) .. controls (312.18,165.61) and (313.6,164.15) .. (315.35,164.15) .. controls (317.11,164.15) and (318.53,165.61) .. (318.53,167.4) .. controls (318.53,169.2) and (317.11,170.65) .. (315.35,170.65) .. controls (313.6,170.65) and (312.18,169.2) .. (312.18,167.4) -- cycle ;
			\draw  [fill={rgb, 255:red, 0; green, 0; blue, 0 }  ,fill opacity=1 ] (276.23,162.53) .. controls (276.23,160.74) and (277.65,159.28) .. (279.4,159.28) .. controls (281.15,159.28) and (282.57,160.74) .. (282.57,162.53) .. controls (282.57,164.32) and (281.15,165.78) .. (279.4,165.78) .. controls (277.65,165.78) and (276.23,164.32) .. (276.23,162.53) -- cycle ;
			\draw  [fill={rgb, 255:red, 0; green, 0; blue, 0 }  ,fill opacity=1 ] (335.71,192.58) .. controls (335.71,190.78) and (337.13,189.33) .. (338.88,189.33) .. controls (340.63,189.33) and (342.05,190.78) .. (342.05,192.58) .. controls (342.05,194.37) and (340.63,195.83) .. (338.88,195.83) .. controls (337.13,195.83) and (335.71,194.37) .. (335.71,192.58) -- cycle ;
			\draw  [fill={rgb, 255:red, 0; green, 0; blue, 0 }  ,fill opacity=1 ] (267.11,121.93) .. controls (267.11,120.13) and (268.53,118.68) .. (270.28,118.68) .. controls (272.03,118.68) and (273.45,120.13) .. (273.45,121.93) .. controls (273.45,123.72) and (272.03,125.17) .. (270.28,125.17) .. controls (268.53,125.17) and (267.11,123.72) .. (267.11,121.93) -- cycle ;
			\draw  [fill={rgb, 255:red, 0; green, 0; blue, 0 }  ,fill opacity=1 ] (144.19,205.57) .. controls (144.19,203.78) and (145.61,202.32) .. (147.36,202.32) .. controls (149.11,202.32) and (150.53,203.78) .. (150.53,205.57) .. controls (150.53,207.36) and (149.11,208.82) .. (147.36,208.82) .. controls (145.61,208.82) and (144.19,207.36) .. (144.19,205.57) -- cycle ;
			\draw   (112.02,135.29) -- (111.7,142.51) -- (104.92,140.48) ;
			\draw   (128.6,188.36) -- (128.45,198.05) -- (119.87,193.97) ;
			\draw   (308.58,142.92) -- (313.28,134.52) -- (318.87,142.33) ;
			\draw   (292.07,206.6) -- (301.33,208.55) -- (295.83,216.43) ;
			\draw   (294.78,123.84) -- (289.8,132.08) -- (284.47,124.07) ;
			\draw    (171.24,171.21) -- (219.08,218.59) ;
			\draw [shift={(220.5,220)}, rotate = 224.72] [color={rgb, 255:red, 0; green, 0; blue, 0 }  ][line width=0.75]    (10.93,-3.29) .. controls (6.95,-1.4) and (3.31,-0.3) .. (0,0) .. controls (3.31,0.3) and (6.95,1.4) .. (10.93,3.29)   ;
			
			\draw    (268.32,172.15) -- (234.65,220.36) ;
			\draw [shift={(233.5,222)}, rotate = 304.93] [color={rgb, 255:red, 0; green, 0; blue, 0 }  ][line width=0.75]    (10.93,-3.29) .. controls (6.95,-1.4) and (3.31,-0.3) .. (0,0) .. controls (3.31,0.3) and (6.95,1.4) .. (10.93,3.29)   ;
			
			\draw    (319.05,124.36) .. controls (349.35,101.1) and (342.76,134.12) .. (340.72,147.75) ;
			\draw [shift={(340.47,149.54)}, rotate = 277.42] [color={rgb, 255:red, 0; green, 0; blue, 0 }  ][line width=0.75]    (10.93,-3.29) .. controls (6.95,-1.4) and (3.31,-0.3) .. (0,0) .. controls (3.31,0.3) and (6.95,1.4) .. (10.93,3.29)   ;
			
			\draw    (340.47,149.54) .. controls (340.47,151.97) and (336.24,158.47) .. (335.71,161.72) ;

			\draw  [dash pattern={on 4.5pt off 4.5pt}]  (335.71,161.72) .. controls (332.36,173.07) and (333.95,167.39) .. (332.36,173.88) ;

			\draw    (332.36,173.88) .. controls (330.14,185.25) and (335.69,188.5) .. (338.88,192.58) ;

			\draw  [fill={rgb, 255:red, 0; green, 0; blue, 0 }  ,fill opacity=1 ] (332.54,164.97) .. controls (332.54,163.17) and (333.96,161.72) .. (335.71,161.72) .. controls (337.46,161.72) and (338.88,163.17) .. (338.88,164.97) .. controls (338.88,166.76) and (337.46,168.21) .. (335.71,168.21) .. controls (333.96,168.21) and (332.54,166.76) .. (332.54,164.97) -- cycle ;
			\draw  [fill={rgb, 255:red, 0; green, 0; blue, 0 }  ,fill opacity=1 ] (117.62,173.9) .. controls (117.62,172.11) and (119.04,170.65) .. (120.79,170.65) .. controls (122.54,170.65) and (123.96,172.11) .. (123.96,173.9) .. controls (123.96,175.69) and (122.54,177.15) .. (120.79,177.15) .. controls (119.04,177.15) and (117.62,175.69) .. (117.62,173.9) -- cycle ;
			\draw  [color={rgb, 255:red, 255; green, 255; blue, 255 }  ,draw opacity=1 ][line width=3] [line join = round][line cap = round] (132.53,174.39) .. controls (132.53,175.24) and (130.99,175.2) .. (130.15,175.2) ;
			\draw  [color={rgb, 255:red, 255; green, 255; blue, 255 }  ,draw opacity=1 ][line width=3] [line join = round][line cap = round] (137.29,182.51) .. controls (138.1,182.51) and (139.67,180.9) .. (139.67,180.07) ;
			
			\draw (232,241) node   {$\mathcal{S}(X) \cap \mathcal{L}_{c}( X)$};
			\draw (110.48,121.93) node   {$\mathfrak{L}(a)$};
			\draw (164.41,194.2) node   {$\mathfrak{L}(b)$};
			\draw (269.09,137.36) node   {$\mathfrak{L}(a)$};
			\draw (323.19,200.51) node   {$\mathfrak{L}(b)$};
			\draw (70,85) node   {$GL^{+}_{c}(X)$};
			\draw (67,237) node   {$GL^{-}_{c}(X)$};
			\draw (354,84) node   {$GL^{+}_{c}(X)$};
			\draw (360,236) node   {$GL^{-}_{c}(X)$};

			\end{tikzpicture}
		\end{center}
		
		\caption{Geometrical interpretation of the parity on $\mathcal{L}_{c}(X)$.}
		\label{Fig3}
	\end{figure}
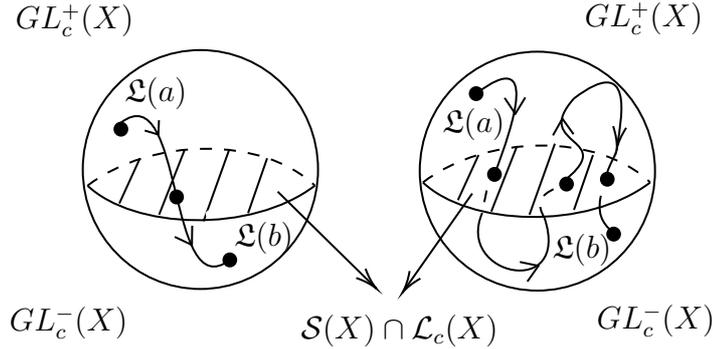
	
The next result, proven by the authors in \cite{LSA}, shows how the parity of any admissible Fredholm path
$\mathfrak{L}\in\mathscr{C}([a,b], \Phi_{0}(X,Y))$ can be computed though the algebraic multiplicity $\chi$.  This result is important from the point of view of the applications.
	
\begin{theorem}
		\label{th3.3}
		Any continuous path $\mathfrak{L}\in\mathscr{C}([a,b],\Phi_{0}(X,Y))$ is $\mathcal{A}$-homotopic to an analytic Fredholm curve $\mathfrak{L}_{\omega}\in\mathscr{C}^{\omega}([a,b],\Phi_{0}(X,Y))$. Moreover, for any of these analytic paths,
		\begin{equation*}
		\sigma(\mathfrak{L},[a,b])=(-1)^{\sum_{i=1}^{n}\chi[\mathfrak{L}_{\omega};\lambda_{i}]},
		\end{equation*}
		where
		\begin{equation*}
		\Sigma(\mathfrak{L}_{\omega})=\{\lambda_{1},\lambda_{2},...,\lambda_{n}\}.
		\end{equation*}
\end{theorem}
	
Subsequently, we consider $\mathfrak{L}\in\mathcal{C}([a,b],\Phi_{0}(X,Y))$ and an isolated eigenvalue $\lambda_{0}\in\Sigma(\mathfrak{L})$. Then, the \textit{localized parity} of $\mathfrak{L}$ at $\lambda_{0}$ is defined through
	\begin{equation*}
	\sigma(\mathfrak{L},\lambda_{0}):=\lim_{\eta\downarrow 0}\sigma(\mathfrak{L},[\lambda_{0}-\eta,\lambda_{0}+\eta]).
	\end{equation*}
As a consequence of Theorem \ref{th3.3}, the next result, going back to \cite{LSA}, holds.

\begin{corollary}
		\label{cr4.6}
		Let $\mathfrak{L}\in\mathcal{C}^{r}([a,b],\Phi_{0}(X,Y))$ with $r\in\mathbb{N}\uplus\{\infty,\omega\}$ and $\lambda_{0}\in\Alg_{k}(\mathfrak{L})$ for some $1\leq k \leq r$. Then
		\begin{equation}\label{c.2}
		\sigma(\mathfrak{L},\lambda_{0})=(-1)^{\chi[\mathfrak{L};\lambda_{0}]}.
		\end{equation}
\end{corollary}
	
The identity \eqref{c.2} establishes the precise relationship between the topological notion of parity and the algebraic concept of multiplicity. The importance of Corollary \ref{cr4.6} relies on the fact that, since the localized parity detects any change of orientation, \eqref{c.2} makes intrinsic to the concept of algebraic multiplicity any change of the local degree.
As the principal difficulty to introduce a topological degree for Fredholm operators of index zero is the absence of orientations in the space of linear isomorphisms $GL(X,Y)\subset\Phi_{0}(X,Y)$, the notion introduced in the next definition, going back to Fitzpatrick, Pejsachowicz and Rabier \cite{FPR},
restricts the space of operators to subsets of $\Phi_{0}(X,Y)$ where is possible to introduce a notion of \emph{orientability}.

\begin{definition}
	\label{de3.5}
	A subset $\mathcal{O}\subset\Phi_{0}(X,Y)$ is said to be \textbf{orientable} if there exists a map $\varepsilon: \mathcal{O}\cap GL(X,Y) \to \mathbb{Z}_{2}$, called orientation, such that
	\begin{equation}
	\label{c.3}
	\sigma(\mathfrak{L},[a,b])=\varepsilon(\mathfrak{L}(a))\cdot \varepsilon(\mathfrak{L}(b)) \quad
	\hbox{for all} \;\; \mathfrak{L}\in \mathscr{C}([a,b],\mathcal{O}).
	\end{equation}
\end{definition}

Observe that if $\mathcal{O}\cap GL(X,Y)=\emptyset$, then $\mathcal{O}$ is trivially orientable. Actually, rather than orienting subsets of $\Phi_{0}(X,Y)$, in \cite{FPR} were oriented, instead,  maps $h:\Lambda\to\Phi_{0}(X,Y)$, where $\Lambda$ is a given topological space. However, if $h(\Lambda)\subset \Phi_{0}(X,Y)$ is orientable in the sense of Definition \ref{de3.5} with orientation $\varepsilon$, necessarily $h:\Lambda\to\Phi_{0}(X,Y)$ is orientable in the sense of \cite{FPR} with orientation $\varepsilon(h(\cdot))$.
\par
Since the parity of a Fredholm curve $\mathfrak{L}$ can be regarded as a generalized local detector of any change of orientation, it is natural to define an orientation $\varepsilon$ of a subset $\mathcal{O}$ of
$\Phi_{0}(X,Y)$ as a map satisfying \eqref{c.3}. Indeed, owing to \eqref{c.3},  $\sigma(\mathfrak{L},[a,b])=-1$ if $\varepsilon(\mathfrak{L}(a))$ and $\varepsilon(\mathfrak{L}(b))$ have contrary sign. Also, note that if $\mathcal{O}$ is an orientable subset of $\Phi_{0}(X,Y)$ with orientation $\varepsilon$, then $\varepsilon$ is locally constant, i.e., $\varepsilon$ is constant on each path connected component of $\mathcal{O} \cap GL(X,Y)$. This is a rather natural property of an orientation.
\par
When $\mathcal{O}\subset \Phi_{0}(X,Y)$ is orientable and $\mathcal{O}\cap GL(X,Y)\neq\emptyset$, then there are, exactly, two different orientations in $\mathcal{O}$. Precisely, given $T\in\mathcal{O} \cap GL(X,Y)$, the two orientations of $\mathcal{O}$ are defined by
\begin{equation}
\label{c.5}
\varepsilon^{\pm}: \mathcal{O} \cap GL(X,Y) \longrightarrow \mathbb{Z}_{2}, \quad
L \mapsto \pm \sigma(\mathfrak{L}_{LT},[a,b])
\end{equation}
where $\mathfrak{L}_{LT}\in \mathcal{C}([a,b],\mathcal{O})$ is an arbitrary  Fredholm path linking $L$ to $T$, and the sign $\pm$ determines the orientation of the path connected component of $T$, i.e., if we choose $\varepsilon^+$, then the orientation of the path connected component of $T$ is $1$, whereas it is
$-1$ if $\varepsilon^-$ is chosen. Lastly, note that any subset of an orientable set of operators is also orientable with the restricted orientation map.
\par
According to \cite{FPR}, any simply connected subset $\mathcal{O}$ of  $\Phi_{0}(X,Y)$
is orientable. So, the set of orientable subsets of $\Phi_{0}(X,Y)$ is really large. More generally, $\mathcal{O}$ is orientable if its $\mathbb{Z}_{2}$-cohomology group $H^{1}(\mathcal{O},\mathbb{Z}_{2})$ is trivial.
\par
The next result, going back to  \cite{LSA}, justifies the geometrical interpretation of the parity as a local detector of change of orientation for the operators of a Fredholm path.

\begin{proposition}
	\label{pr3.7}
	Let $\mathcal{O}$ be an orientable subset of $\Phi_{0}(X,Y)$ and $\mathfrak{L}\in\mathscr{C}([a,b],\mathcal{O})$. Then, $\sigma(\mathfrak{L},[a,b])=-1$ if, and only if, $\mathfrak{L}(a)$ and $\mathfrak{L}(b)$ lye in different path connected components of $\mathcal{O}\cap GL(X,Y)$ with opposite orientations.
\end{proposition}

Finally, the next result, going back as well to \cite{LSA},
reduces the problem of detecting any change of orientation
to the problem of the computation of the local multiplicity. It allows to interpret the algebraic multiplicity as a local detector of change of orientation for the operators of a Fredholm path.
Given $\mathfrak{L}\in\mathcal{C}([a,b],\Phi_{0}(X,Y))$ and $\delta>0$, an isolated eigenvalue $\lambda_{0}\in\Sigma(\mathfrak{L})$ is said to be \emph{$\delta$-isolated} if
$$
\Sigma(\mathfrak{L})\cap[\lambda_{0}-\delta,\lambda_{0}+\delta]=\{\lambda_{0}\}.
$$

\begin{theorem}
	\label{th3.8}
Suppose that $\mathcal{O}\subset\Phi_{0}(X,Y)$ is an orientable subset, $\mathfrak{L}\in\mathcal{C}([a,b],\mathcal{O})$  is a Fredholm curve and $\lambda_{0}\in\Sigma(\mathfrak{L})$ a $\delta$-isolated eigenvalue. Then, the next assertions are equivalent:
	\begin{enumerate}
		\item[{\rm (a)}] $\sum_{\lambda\in\Sigma(\mathfrak{L}_{\omega})}\chi[\mathfrak{L}_{\omega};\lambda_{0}]$ is odd for any analytical path $\mathfrak{L}_{\omega}\in \mathscr{C}^{\omega}([\lambda_{0}-\delta,\lambda_{0}+\delta],\Phi_{0}(X,Y))$ such that $\mathfrak{L}\vert_{[\lambda_{0}-\delta,\lambda_{0}+\delta]}$ and $\mathfrak{L}_\omega$ are $\mathcal{A}$-homotopic.
		\item[{\rm (b)}] $\mathfrak{L}(\lambda_{0}-\delta)$ and $\mathfrak{L}(\lambda_{0}+\delta)$ live in different path-connected components of $\mathcal{O}\cap GL(X,Y)$ with opposite orientations.
	\end{enumerate}
\end{theorem}

\section{Axiomatization and Uniqueness of the Parity}

The aim of this section is axiomatizing the concept of \emph{parity}. So, establishing its uniqueness. Our
axiomatization is based upon Theorem \ref{th2.1} and Corollary \ref{cr4.6}. Thanks to this axiomatization, we are establishing the uniqueness of a local detector of change of orientability. We will begin by  axiomatizing the parity as a local object. Then, we will do it in a global setting.
\par
Subsequently, for any interval $\mathcal{I}\subset \mathbb{R}$ and $\lambda_{0}\in \text{Int\,}\mathcal{I}$, we will denote by $\mathcal{C}^{\omega}_{\lambda_{0}}(\mathcal{I},\Phi_{0}(X,Y))$ the space of all the analytic paths  $\mathfrak{L}\in \mathcal{C}^{\omega}(\mathcal{I},\Phi_{0}(X,Y))$ such that $\mathfrak{L}(\lambda)\in GL(X,Y)$ for $\lambda\neq\lambda_{0}$.

\begin{theorem}
\label{th4.1}
For every $\varepsilon>0$ and $\lambda_{0}\in\mathbb{R}$, there exists a unique $\mathbb{Z}_{2}$-valued map $$
   \sigma(\cdot,\lambda_{0}):\;\; \mathcal{C}^\omega_{\lambda_0}\equiv  \mathcal{C}^{\omega}_{\lambda_0}((\lambda_{0}-\varepsilon,\lambda_{0}+\varepsilon),\Phi_{0}(X))
   \longrightarrow\mathbb{Z}_{2}
$$
such that:
\begin{enumerate}
		\item[{\rm (N)}]  \textbf{Normalization:}  $\sigma(\mathfrak{L},\lambda_{0})=1$ if $\mathfrak{L}(\lambda_{0})\in GL(X)$, and there exists a rank one projection $P_{0}\in\mathcal{L}(X)$ such that $\sigma(\mathfrak{E},\lambda_{0})=-1$ where
\begin{equation*}
		\mathfrak{E}(\lambda):=(\lambda-\lambda_{0})P_{0}+I_{X}-P_{0}.
\end{equation*}
		
\item[{\rm (P)}]  \textbf{Product Formula:} For every  $\mathfrak{L}, \mathfrak{M}\in\mathcal{C}^{\omega}_{\lambda_{0}}$,
\begin{equation*} \sigma(\mathfrak{L}\circ\mathfrak{M},\lambda_{0})=\sigma(\mathfrak{L},\lambda_{0})
\cdot\sigma(\mathfrak{M},\lambda_{0}).
\end{equation*}
\end{enumerate}
Moreover, for every $\mathfrak{L}\in\mathcal{C}^{\omega}_{\lambda_{0}}$, the parity map is given by
\begin{equation*}
	\sigma(\mathfrak{L},\lambda_{0})=(-1)^{\chi[\mathfrak{L};\lambda_{0}]}.
\end{equation*}
\end{theorem}

\begin{proof}
First, we will prove that,  for every rank one projection $P\in\mathcal{L}(X)$, setting
$$
  \mathfrak{F}(\lambda)=(\lambda-\lambda_{0})P+I_{X}-P,
$$
one has that
\begin{equation}
\label{d.1}
   \sigma(\mathfrak{F},\lambda_{0})=-1.
\end{equation}
Indeed,
By Lemma 6.1.1 of \cite{LGMC}, there exists $T\in GL(X)$ such that $P=T^{-1}P_{0}T$. Thus,
$$
  \mathfrak{F}(\lambda)=T^{-1}[(\lambda-\lambda_{0})P_{0}+I_{X}-P_{0}]T=T^{-1}\mathfrak{E}(\lambda)T
$$
and hence, by axioms (P) and (N),
\begin{equation*}
 \sigma(\mathfrak{L},\lambda_{0})=\sigma(T^{-1},\lambda_{0})\cdot
 \sigma(\mathfrak{E}(\lambda),\lambda_{0})\cdot\sigma(T,\lambda_{0})=-1.
\end{equation*}
On the other hand, for any given $\mathfrak{L}\in\mathcal{C}^{\omega}_{\lambda_{0}}$, by Corollary 5.3.2(b) of \cite{LGMC}, which goes back to the proof of Theorem 5.3.1 of \cite{LG01}, there exist $k$ finite-rank projections $\Pi_{0},\Pi_{2},\cdots,\Pi_{k-1}\in\mathcal{L}(X)$ and a (globally invertible) path  $\mathfrak{I}\in\mathcal{C}^{\omega}((\lambda_{0}-\varepsilon,\lambda_{0}+\varepsilon),GL(X))$
such that, setting
\begin{equation*}
	\mathfrak{C}_{\Pi_{i}}(t):=t\, \Pi_{i}+I_{X}-\Pi_{i},\qquad i\in\{0,1,...,k-1\}, \;\; t\in\mathbb{R},
\end{equation*}
one has that
\begin{equation*}
\mathfrak{L}(\lambda)=\mathfrak{I}(\lambda)\circ\mathfrak{C}_{\Pi_{0}}
(\lambda-\lambda_{0})\circ\mathfrak{C}_{\Pi_{1}}(\lambda-\lambda_{0})\circ \cdots\circ
\mathfrak{C}_{\Pi_{k-1}}(\lambda-\lambda_{0}).
\end{equation*}
Moreover, for every $i\in\{0,1,...,k-1\}$, there are $r_{i}=\text{rank}\,\Pi_{i}$ projections of rank one, $P_{j,i}$, $1\leq j \leq r_{i}$, such that
\begin{equation*}	 \mathfrak{C}_{\Pi_{i}}=\mathfrak{C}_{P_{1,i}}\circ\mathfrak{C}_{P_{2,i}}\circ\cdots\circ
\mathfrak{C}_{P_{r_{i},i}}.
\end{equation*}
Consequently, we find from the axiom (P) and \eqref{d.1} that
\begin{equation*}	
\sigma(\mathfrak{C}_{\Pi_{i}},\lambda_{0})=
\sigma(\mathfrak{C}_{P_{1,i}},\lambda_{0})\cdots\sigma(\mathfrak{C}_{P_{r_{i},i}},\lambda_{0})
=(-1)^{r_{i}}=(-1)^{\mathrm{rank}\,\Pi_{i}}
\end{equation*}
and, therefore, applying again the axiom (P) yields
\begin{equation*}	
\sigma(\mathfrak{L},\lambda_{0})=\sigma(\mathfrak{J},\lambda_{0})
\cdot\sigma(\mathfrak{C}_{\Pi_{0}},\lambda_{0})\cdots\sigma(\mathfrak{C}_{\Pi_{k-1}},\lambda_{0})
=(-1)^{\sum_{i=0}^{k-1}{\mathrm{rank}\,\Pi_{i}}}.
\end{equation*}
Finally, since owing to Corollary 5.3.2 of \cite{LGMC}, we  have that
$$
\chi[\mathfrak{L};\lambda_{0}]=\sum_{i=0}^{k-1}{\mathrm{rank}\,\Pi_{i}},
$$
it becomes apparent that
\begin{equation*}
	\sigma(\mathfrak{L},\lambda_{0})=(-1)^{\chi[\mathfrak{L};\lambda_{0}]}.
\end{equation*}
This concludes the proof.
\end{proof}

Once the local parity is determined, we will give the global axiomatization. A pair $(\mathfrak{L},[a,b])$ is said to be \emph{admissible} if $\mathfrak{L}\in\mathscr{C}([a,b],\Phi_{0}(X))$. The set of admissible pairs will be denoted by $\mathscr{A}$.

\begin{theorem}
\label{th4.2}
	There exists a unique $\mathbb{Z}_{2}$-valued map $\sigma:\mathscr{A}\to\mathbb{Z}_{2}$ such that
\begin{enumerate}
		\item[{\rm (N)}]  \textbf{Normalization:} For every $\mathfrak{L}\in\mathcal{C}^{\omega}_{\lambda_0}([\lambda_{0}-\eta,\lambda_{0}+\eta],\Phi_{0}(X))$,
\begin{equation*}		 \sigma(\mathfrak{L},[\lambda_{0}-\eta,\lambda_{0}+\eta])=(-1)^{\chi[\mathfrak{L};\lambda_{0}]}.
\end{equation*}
		
\item[{\rm (P)}]  \textbf{Product Formula:} For every $(\mathfrak{L},[a,b])\in\mathscr{A}$ and $c\in(a,b)$ such that $c\notin\Sigma(\mathfrak{L})$,
\begin{equation*}
		\sigma(\mathfrak{L},[a,b])=\sigma(\mathfrak{L},[a,c])\cdot\sigma(\mathfrak{L},[c,b]).
\end{equation*}
		
\item[{\rm (H)}]  \textbf{Homotopy Invariance:} For every homotopy  $H\in\mathcal{C}([0,1]\times[a,b],\Phi_{0}(X))$
    such that $(H(t,\cdot),[a,b])\in\mathscr{A}$ for all $t\in [0,1]$,
\begin{equation*}
		\sigma(H(0,\cdot),[a,b])=\sigma(H(1,\cdot),[a,b]).
\end{equation*}		
\end{enumerate}
Moreover, $\sigma(\mathfrak{L},[a,b])$ equals the parity map introduced by Fitzpatrick and Pejsachowitz in \cite{FP2}.
\end{theorem}

\begin{proof}
Pick $(\mathfrak{L},[a,b])\in\mathscr{A}$.  By Theorem \ref{th3.3}, we already know that there exists an analytic curve $\mathfrak{L}_{\omega}\in\mathscr{C}^{\omega}([a,b],\Phi_{0}(X))$ $\mathcal{A}$-homotopic to $\mathfrak{L}$, i.e., there exists $H\in\mathcal{C}([0,1]\times[a,b],\Phi_{0}(X))$ such that  $H([0,1]\times\{a,b\})\subset GL(X)$,
$$
  H(0,\lambda)=\mathfrak{L}(\lambda) \quad\hbox{and}\quad  H(1,\lambda)=\mathfrak{L}_\omega(\lambda)
  \quad \hbox{for all}\;\;\lambda \in [a,b].
$$
Then, by the axiom (H),
\begin{equation}
\label{d.2}
	\sigma(\mathfrak{L},[a,b])=\sigma(\mathfrak{L}_{\omega},[a,b]).
\end{equation}
Suppose that $\Sigma(\mathfrak{L}_{\omega})\cap [a,b]=\emptyset$ and pick any $\lambda_0\in (a,b)$. Then, since $\chi[\mathfrak{L}_\omega;\lambda_0]=0$, it follows from (N)
that
$$
  \sigma(\mathfrak{L}_{\omega},[a,b])=(-1)^{\chi[\mathfrak{L}_\omega;\lambda_0]}=1.
$$
Therefore, \eqref{d.2} implies that
\begin{equation*}
	\sigma(\mathfrak{L},[a,b])=1.
\end{equation*}
Now, suppose that $\Sigma(\mathfrak{L}_{\omega})\neq\emptyset$. Since $\mathfrak{L}_\omega(a)\in GL(X)$,
it follows from Theorems 4.4.1 and 4.4.4 of \cite{LG01} that $\Sigma(\mathfrak{L}_{\omega})$ is discrete.
Thus,
$$
   \Sigma(\mathfrak{L}_\omega)=\{\lambda_{1},\lambda_{2},....,\lambda_{n}\}
$$
for some
$$
   a<\lambda_1<\lambda_2<\cdots<\lambda_n<b.
$$
Let $\varepsilon>0$ be sufficiently small so that $\lambda_i$ is $\varepsilon$-isolated for all $i\in\{1,...,n\}$. Then, by the axioms (P) and (N), we find that
\begin{align*}	
\sigma(\mathfrak{L}_{\omega},[a,b]) & =
\prod_{i=1}^{n}\sigma(\mathfrak{L}_{\omega},[\lambda_{i}-\varepsilon,\lambda_{i}+\varepsilon])\\ & = \prod_{i=1}^n (-1)^{\chi[\mathfrak{L}_{\omega},\lambda_{i}]} =(-1)^{\sum_{i=1}^{n}\chi[\mathfrak{L}_{\omega};\lambda_{i}]}.
\end{align*}
Therefore, by \eqref{d.2}, it becomes apparent that
\begin{equation*}
	\sigma(\mathfrak{L},[a,b])=(-1)^{\sum_{i=1}^{n}\chi[\mathfrak{L}_{\omega};\lambda_{i}]}.
\end{equation*}
Consequently, owing to Theorem \ref{th3.3}, the map
$\sigma:\mathscr{A}\to\mathbb{Z}_{2}$ is the parity defined by Fitzpatrick and Pejsachowitz in \cite{FP2}. This concludes the proof.
\end{proof}

Note that the normalization property (N) in Theorem \ref{th4.2} is determined by the local uniqueness
of the parity provided by Theorem \ref{th4.1}.

\section{Axiomatization and Uniqueness of the Topological Degree}

The aim of this section is delivering the proof of Theorem \ref{th1.2}. We begin by recalling our main theorem. As already discussed in Section 1, for any open and bounded subset, $\Omega$, of a Banach space $X$, an operator $f:\overline{\Omega}\subset X \to Y$ is said to be $\mathcal{C}^{1}$-\emph{Fredholm of index zero} if $f\in \mathcal{C}^{1}(\overline{\Omega},Y)$ and $Df\in \mathcal{C}(\overline{\Omega},\Phi_{0}(X,Y))$, and the set of all these operators is denoted by $\mathscr{F}^{1}_{0}(\Omega,Y)$. An operator $f\in\mathscr{F}^{1}_{0}(\Omega,Y)$ is said to be \emph{orientable} when  $Df(\Omega)$ is an orientable subset of $\Phi_{0}(X,Y)$. Moreover, for any open and bounded subset, $\Omega$, of a Banach space $X$ and any operator $f:\overline{\Omega}\subset X \to Y$ satisfying 	
\begin{enumerate}
	\item $f\in \mathscr{F}^{1}_{0}(\Omega,Y)$ is \emph{orientable} with orientation $\varepsilon$,
	\item $f$ is \emph{proper} in $\overline{\Omega}$, i.e., $f^{-1}(K)$ is compact for every compact subset $K\subset Y$,
	\item $0\notin f(\partial \Omega)$,
\end{enumerate}
it is said that $(f,\Omega,\varepsilon)$ is a \emph{Fredholm admissible tern}. The set of all
Fredholm admissible tern is denoted by $\mathscr{A}$. Given $(f,\Omega,\varepsilon)\in \mathscr{A}$,
it is said that $(f,\Omega,\varepsilon)$ is a regular tern if $0$ is a regular value of $f$, i.e.,
$Df(x)\in GL(X,Y)$ for all $x \in f^{-1}(0)$. The set of regular terns is denoted by $\mathscr{R}$.
Lastly, a map $H\in \mathcal{C}^{1}([0,1]\times \overline{\Omega},Y)$ is said to be $\mathcal{C}^{1}$-\textit{Fredholm homotopy} if $D_{x}H(t,\cdot)\in\Phi_{0}(X,Y)$ for each $t\in[0,1]$ and it is called \textit{orientable} if $D_{x}H([0,1]\times\Omega)$ is an orientable subset of $\Phi_{0}(X,Y)$. Hereinafter, the notation $\varepsilon_{t}$ stands for the restriction of $\varepsilon$ to $D_{x}H(\{t\}\times\Omega)\cap GL(X,Y)$ for each $t\in[0,1]$. Theorem \ref{th1.2} reads as follows.

\begin{theorem}
	\label{th5.1}
	There exists a unique integer valued map $\deg: \mathscr{A}\to \mathbb{Z}$
	satisfying the next properties:
	\begin{enumerate}
\item[{\rm (N)}] \textbf{Normalization:} $\deg(L,\Omega,\varepsilon)=\varepsilon(L)$ for all
		$L\in GL(X,Y)$ if $0\in \Omega$.
		
\item[{\rm (A)}] \textbf{Additivity:} For every $(f,\Omega,\varepsilon)\in\mathscr{A}$  and any
		pair of disjoint open subsets $\Omega_{1}$ and $\Omega_{2}$ of $\Omega$ with $0\notin f(\Omega\backslash (\Omega_{1}\uplus \Omega_{2}))$,
\begin{equation*}
 \deg(f,\Omega,\varepsilon)=\deg(f,\Omega_{1},\varepsilon)+\deg(f,\Omega_{2},\varepsilon).
\end{equation*}
		
\item[{\rm (H)}] \textbf{Homotopy Invariance:} For each proper $\mathcal{C}^{1}$-Fredholm homotopy
		$H\in \mathcal{C}^{1}([0,1]\times \overline{\Omega}, Y)$ with orientation $\varepsilon$ and $(H(t,\cdot),\Omega,\varepsilon_{t})\in\mathscr{A}$ for each $t\in[0,1]$
\begin{equation*}
		\deg(H(0,\cdot),\Omega,\varepsilon_{0})=\deg(H(1,\cdot),\Omega,\varepsilon_{1}).
		\end{equation*}
\end{enumerate}

\noindent Moreover, for every $(f,\Omega,\varepsilon)\in\mathscr{R}$ satisfying $Df(\Omega)\cap GL(X,Y)\neq\emptyset$ and $L\in Df(\Omega)\cap GL(X,Y)$,
\begin{equation}
\label{5.1}
\deg(f,\Omega,\varepsilon)=\varepsilon(L)\cdot\sum_{x\in f^{-1}(0)\cap \Omega} (-1)^{\sum_{\lambda_{x}\in\Sigma(\mathfrak{L}_{\omega,x})}\chi[\mathfrak{L}_{\omega,x};\lambda_{x}]}
\end{equation}
where $\mathfrak{L}_{\omega,x}\in \mathscr{C}^{\omega}([a,b],\Phi_{0}(X,Y))$ is an analytical path $\mathcal{A}$-homotopic to some path $\mathfrak{L}_{x}\in\mathcal{C}([a,b],Df(\Omega))$ joining $Df(x)$ and $L$.
\end{theorem}

Applying axiom (A) with $\Omega_1=\Omega_2=\Omega =\emptyset$, it becomes apparent that $\mathrm{deg}(f,\emptyset,\varepsilon)=0$. If $(f,\Omega,\varepsilon)\in\mathscr{A}$ and $f^{-1}(0)\cap\Omega=\emptyset$, applying again (A) with $\Omega_1=\Omega_2=\emptyset$, we find that $\mathrm{deg}(f,\Omega,\varepsilon)=0$. Equivalently, $f$ admits a zero in
$\Omega$ if   $\mathrm{deg}(f,\Omega,\varepsilon)\neq 0$.
\par
As already commented in Section 1, as the existence of  $\deg$
goes back to Fitzpatrick, Pejsachowicz and Rabier \cite{FPR}, and the formula \eqref{5.1}
was proven  by the authors in \cite{LSA}, Theorem \ref{th5.1} actually establishes the uniqueness of $\deg$.
\par
To prove the uniqueness, it is appropriate to sketch briefly the construction of $\deg$ carried over in \cite{FPR}. Let  $(f,\Omega,\varepsilon)\in\mathscr{A}$. By definition, $f\in\mathscr{F}^{1}_{0}(\Omega,Y)$ is $\mathcal{C}^{1}$-Fredhom of index zero and it is $\varepsilon$-orientable, i.e., $Df(\Omega)$ is an orientable subset of $\Phi_{0}(X,Y)$ with orientation
$$
\varepsilon:Df(\Omega)\cap GL(X,Y) \longrightarrow \mathbb{Z}_{2}.
$$
Once an orientation has been defined in $Df(\Omega)$,  the degree $\deg$ can be defined
as the Leray--Schauder degree $\deg_{LS}$ as soon as $0$ is a regular value of $f$, because since in such case $f^{-1}(0)\cap \Omega$ is finite, possibly empty, one can define, in complete agreement with the axioms (N), (A) and (H),
\begin{equation}
\label{5.2}
\deg(f,\Omega,\varepsilon):=\sum_{x\in f^{-1}(0)\cap \Omega} \varepsilon(Df(x)).
\end{equation}
If $f^{-1}(0)\cap \Omega=\emptyset$, as we already mentioned, $\mathrm{deg}(f,\Omega,\varepsilon)=0$. When $0$ is not a regular value, then, by definition,
\begin{equation*}
\deg(f,\Omega,\varepsilon):=\deg(f-x_{0},\Omega,\varepsilon)
\end{equation*}
where $x_{0}$ is any regular value of $f$ lying in a sufficiently small neighborhood of $0$. The existence of such regular values is guaranteed by a theorem of Quinn and Sard \cite{QS}, a version of the Sard--Smale Theorem, \cite{Sm},  not requiring the separability of the involved Banach spaces.
\par
It is worth-mentioning that $\deg$ extends $\deg_{LS}$ to this more general setting if we consider the Leray--Schauder admissible pairs to be in $\mathscr{A}_{LS}^{1}$. Indeed,
$\mathcal{L}_{c}(X)$ is simply connected and hence, according to \cite{FPR}, orientable. Moreover, the maps $\varepsilon^{\pm}: GL_{c}(X)\to\mathbb{Z}_{2}$ defined by
\begin{equation}
\label{5.3}
\varepsilon^{\pm}(L)=\pm\deg_{LS}(L,\Omega),
\end{equation}
where the right hand of equation \eqref{5.3} is defined by \eqref{b.4}, determine the two orientations of $\mathcal{L}_{c}(X)$. Therefore, if for every  Leray--Schauder regular pair $(f,\Omega)$ the orientation of $f$ is chosen as the restriction of $\varepsilon^{+}$ to $Df(\Omega)\cap GL(X)$, then $(f,\Omega,\varepsilon^{+})\in\mathscr{R}$ and, thanks to \eqref{5.2} and \eqref{5.3},
\begin{align*}
\mathrm{deg}_{LS}(f,\Omega) & =\sum_{x\in f^{-1}(0)\cap \Omega} \mathrm{deg}_{LS}(Df(x),\Omega)\\ & =\sum_{x\in f^{-1}(0)\cap \Omega}\varepsilon^{+}(Df(x))=\deg(f,\Omega,\varepsilon^{+})
\end{align*}
and therefore,
\begin{equation*}
\deg(f,\Omega,\varepsilon^{+})=\deg_{LS}(f,\Omega),
\end{equation*}
as claimed above. We have all necessary ingredients to prove Theorem \ref{th5.1}. It suffices to prove the uniqueness.

\begin{proof}[Proof of the Uniqueness]
We first prove that, for every $(f,\Omega,\varepsilon)\in\mathscr{R}$, the topological degree is given by \eqref{5.1} if $Df(\Omega)\cap GL(X,Y)\neq\emptyset$ and $\deg(f,\Omega,\varepsilon)=0$ if $Df(\Omega)\cap GL(X,Y)=\emptyset$. Pick $(f,\Omega,\varepsilon)\in\mathscr{R}$. Then, $f^{-1}(0)\cap\Omega$ is finite, possibly empty. If it is empty, then, applying axiom (A) with $\Omega_1=\Omega_2=\Omega =\emptyset$, it becomes apparent that $\mathrm{deg}(f,\emptyset,\varepsilon)=0$. Thus, applying again (A) with $\Omega_1=\Omega_2=\emptyset$, we find that
$$
  \mathrm{deg}(f,\Omega,\varepsilon)=0.
$$
If $Df(\Omega)\cap GL(X,Y)=\emptyset$, necessarily $f^{-1}(0)\cap\Omega=\emptyset$ and therefore $\deg(f,\Omega,\varepsilon)=0$ as required. Now suppose that, for some $n\geq 1$ and $x_i\in\Omega$, $i\in\{1,...,n\}$,
$$
  f^{-1}(0)\cap\Omega=\{x_{1},x_{2},...,x_{n}\}.
$$
Since $0$ is a regular value of $f$, by the inverse function theorem, $f\vert_{B_{\eta_{i}}(x_{i})}$ is a diffeomorphism for each $i\in\{1,2,...,n\}$ for sufficiently small $\eta_{i}>0$ and therefore by axiom (A)
\begin{equation}
\label{tete}
	\deg(f,\Omega,\varepsilon)=\sum_{i=1}^{n}\deg(f,B_{\eta_{i}}(x_{i}),\varepsilon).
\end{equation}
Subsequently, we fix $i\in\{1,2,...,n\}$ and consider the homotopy $H_{i}$ defined by
\begin{equation*}
	\begin{array}{lccl}
	H_{i}: & [0,1]\times B_{\eta_{i}}(x_{i}) & \longrightarrow & Y \\
	 & (t,x) & \mapsto & tf(x)+(1-t)Df(x_{i})(x-x_{i}).
	\end{array}
\end{equation*}
The next result of technical nature holds.

\begin{lemma}
\label{le5.2}
$H_{i}\in\mathcal{C}^{1}([0,1]\times\overline{B_{\tau_i}(x_i)},Y)$ and it is proper for sufficiently small
$\tau_i>0$.
\end{lemma}
\begin{proof}
Obviously, $H_{i}\in\mathcal{C}^{1}([0,1]\times B_{\eta_{i}}(x_{i}),Y)$ and
\begin{equation*}
	D_{x}H_{i}(t,\cdot)=tDf(\cdot)+(1-t)Df(x_{i})=Df(x_{i})+t(Df(\cdot)-Df(x_{i})).
\end{equation*}
Since $Df\in\mathcal{C}(\Omega,\Phi_{0}(X,Y))$, $Df(x_{i})\in GL(X,Y)$, and $GL(X,Y)$ is open,  we have that, for sufficiently small $\eta_i>0$,
\begin{equation}
\label{tata}
   D_{x}H_{i}(t,B_{\eta_{i}}(x_{i}))\subset GL(X,Y) \quad \hbox{for all} \;\, t \in [0,1].
\end{equation}
In particular, $D_{x}H_{i}(t,x)\in \Phi_0(X,Y)$ for all $(t,x)\in [0,1]\times B_{\eta_{i}}(x_{i})$.
Thus, by definition, $H_{i}(t,\cdot)\in\mathscr{F}^{1}_{0}(\Omega,Y)$ for all $t\in [0,1]$. This also entails that $DH_{i}$ is a Fredholm operator of index one from $\mathbb{R}\times X$ to $Y$. Thus, by Theorem 1.6 of Smale \cite{Sm}, $H_{i}$ is locally proper, i.e., for every $t\in[0,1]$, there exists an open interval containing $t$, $\mathcal{I}(t)\subset [0,1]$, and an open ball centered in $x_{i}$ with radius $\delta_{t}$, $B_{\delta_t}(x_i)$, such that $H_{i}$ is proper in $\overline{\mathcal{I}(t)}\times\overline{B_{\delta_t}(x_i)}$. In particular
\begin{equation*}
	[0,1]\times\{x_{i}\}\subset \bigcup_{t\in[0,1]}\mathcal{I}(t)\times B_{\delta_t}(x_i).
\end{equation*}
Since $[0,1]\times\{x_{i}\}$ is compact, there exists a finite subset $\{t_{1},t_{2},...,t_{n}\}\subset[0,1]$ such that
\begin{equation*}
	[0,1]\times\{x_{i}\}\subset \bigcup_{j=1}^{n}\mathcal{I}(t_{j})\times B_{\delta_{_{t_j}}}(x_{i}),
\end{equation*}
as illustrated in Figure \ref{Fig5}.
Let
$$
   \delta_{i}:=\min\{\delta_{t_{1}},\delta_{t_{2}},...,\delta_{t_{n}}\}.
$$
Then,
\begin{equation*}
	[0,1]\times\{x_{i}\}\subset [0,1]\times B_{\delta_{i}}(x_{i}).
\end{equation*}
Let $\tau_{i}<\min\{\eta_{i},\delta_{i}\}$. Then, $H_{i}$ is proper in $\overline{\mathcal{I}(t)}\times\overline{B_{\tau_{i}}(x_{i})}$ since the restriction of a proper map to a closed subset is also proper. On the other hand, since
\begin{equation*}
[0,1]\times\overline{B_{\tau_{i}}(x_{i})}=\bigcup_{i=1}^{n}\overline{\mathcal{I}(t_{i})}\times\overline{B_{\tau_{i}}(x_{i})}
\end{equation*}
and each $\overline{\mathcal{I}(t_{i})}\times\overline{B_{\tau_{i}}(x_{i})}$ is closed, necessarily $H_{i}$ is proper on $[0,1]\times\overline{B_{\tau_{i}}(x_{i})}$. Therefore, $$
  H_{i}\in\mathcal{C}^{1}([0,1]\times\overline{B_{\tau_{i}}(x_{i})},Y)
$$
and it is proper. The proof is complete.
\end{proof}
	
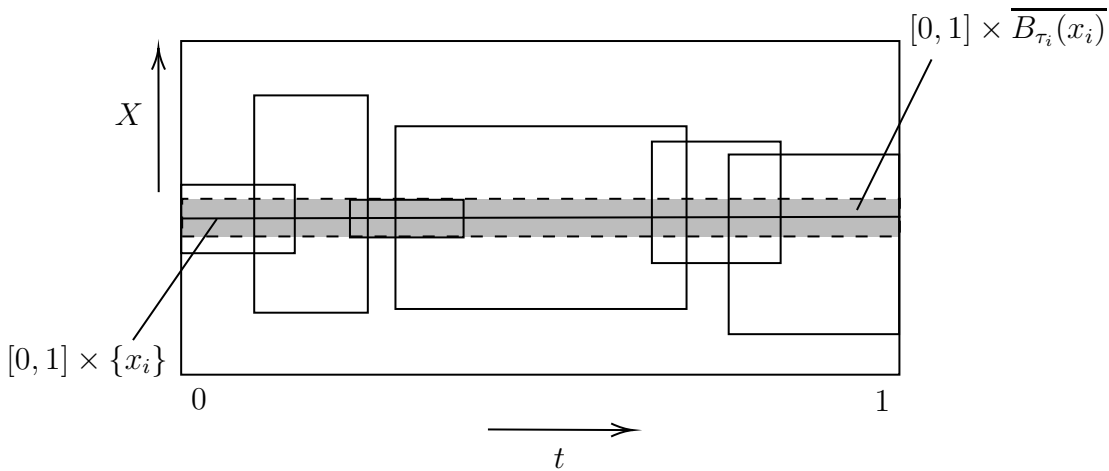
\begin{figure}
	\begin{center}

		\tikzset{every picture/.style={line width=0.75pt}} 
		
		\begin{tikzpicture}[x=0.75pt,y=0.75pt,yscale=-1,xscale=1]
		
		\draw   (124.08,91.36) -- (482.44,91.36) -- (482.44,259.21) -- (124.08,259.21) -- cycle ;
		\draw    (124.23,180.63) -- (483,179.77) ;

		\draw   (124.08,163.66) -- (180.77,163.66) -- (180.77,198.1) -- (124.08,198.1) -- cycle ;
		\draw   (160.52,118.71) -- (217.21,118.71) -- (217.21,228.03) -- (160.52,228.03) -- cycle ;
		\draw   (208.3,171.22) -- (264.99,171.22) -- (264.99,190.16) -- (208.3,190.16) -- cycle ;
		\draw   (231,134.24) -- (376.23,134.24) -- (376.23,226.16) -- (231,226.16) -- cycle ;
		\draw   (358.94,141.95) -- (423.18,141.95) -- (423.18,203.07) -- (358.94,203.07) -- cycle ;
		\draw   (397.26,148.45) -- (482.3,148.45) -- (482.3,238.83) -- (397.26,238.83) -- cycle ;
		\draw  [fill={rgb, 255:red, 0; green, 0; blue, 0 }  ,fill opacity=0.26 ][dash pattern={on 4.5pt off 4.5pt}] (124.58,170.73) -- (482.65,170.73) -- (482.65,189.67) -- (124.58,189.67) -- cycle ;
		\draw    (498.49,100.63) -- (461.24,176.38) ;

		\draw    (100.04,241.81) -- (141.89,180.88) ;

		\draw    (112.87,167.39) -- (112.74,96.8) ;
		\draw [shift={(112.74,94.8)}, rotate = 449.89] [color={rgb, 255:red, 0; green, 0; blue, 0 }  ][line width=0.75]    (10.93,-3.29) .. controls (6.95,-1.4) and (3.31,-0.3) .. (0,0) .. controls (3.31,0.3) and (6.95,1.4) .. (10.93,3.29)   ;
		
		\draw    (277.14,286.76) -- (347.55,287.08) ;
		\draw [shift={(349.55,287.09)}, rotate = 180.26] [color={rgb, 255:red, 0; green, 0; blue, 0 }  ][line width=0.75]    (10.93,-3.29) .. controls (6.95,-1.4) and (3.31,-0.3) .. (0,0) .. controls (3.31,0.3) and (6.95,1.4) .. (10.93,3.29)   ;

		\draw (536.39,84.19) node    {$[ 0,1] \times \overline{B_{\tau _{i}}( x_{i})}$};
		\draw (77.01,252.19) node    {$[ 0,1] \times \{x_{i}\}$};
		\draw (132.94,271.12) node    {$0$};
		\draw (473.89,272.12) node    {$1$};
		\draw (98.16,128.37) node    {$X$};
		\draw (312.77,301.39) node    {$t$};

		\end{tikzpicture}
	\end{center}

		\caption{Scheme of the construction.}
		\label{Fig5}
\end{figure}

\begin{lemma}
\label{le5.3}
 $0\notin H_{i}(t,\partial B_{\tau_{i}}(x_i))$ for each $t\in[0,1]$ and sufficiently small $\tau_i>0$.
\end{lemma}
\begin{proof}
On the contrary, assume that $0\in H_{i}(t,\partial B_{\tau_{i}}(x_i))$ for some $t\in[0,1]$ and $\tau_i<\eta_i$, i.e., there exists $x\in\partial B_{\tau_{i}}(x_i)$ such that $H_{i}(t,x)=0$. Thus,
\begin{equation}
\label{5.4}
  t[f(x)-Df(x_{i})(x-x_{i})]+Df(x_{i})(x-x_{i})=0.
\end{equation}
Necessarily $t>0$, because $t=0$ implies $Df(x_{i})(x-x_{i})=0$ and in such case $Df(x_i)$ cannot be an
isomorphism. So, dividing \eqref{5.4} by $t$ yields
\begin{equation*}
	f(x)-Df(x_{i})(x-x_{i})=-\frac{1}{t}Df(x_{i})(x-x_{i}).
\end{equation*}
Taking norms and dividing by $\|x-x_{i}\|>0$, we obtain that
\begin{equation}
\label{5.5} 	 \frac{\|f(x)-Df(x_{i})(x-x_{i})\|}{\|x-x_{i}\|}=\frac{1}{t}\Big\|Df(x_{i})\Big(\frac{x-x_{i}}{\|x-x_{i}\|}
\Big)\Big\|\geq \frac{1}{t}\inf_{\|x\|=1}\|Df(x_{i})(x)\|.
\end{equation}
Since $Df(x_{i})\in GL(X,Y)$ and $\partial B_{1}(x_i)$ is closed, $Df(x_{i})(\partial B_{1}(x_i))$ is closed. Hence,
$$
  m\equiv \inf_{\|x\|=1}\|Df(x_{i})(x)\|
$$
is attained and, since $0\notin Df(x_{i})(\partial B_{1}(x_i))$, we find that $m>0$ and therefore,
it follows from \eqref{5.5} that
\begin{equation}
\label{5.6}
	\frac{\|f(x)-Df(x_{i})(x-x_{i})\|}{\|x-x_{i}\|}\geq \frac{m}{t}>0.
\end{equation}
On the other hand, since $f$ is differentiable at $x_i$ and $f(x_{i})=0$
\begin{equation*}
	\frac{\|f(x)-Df(x_{i})(x-x_{i})\|}{\|x-x_{i}\|}\xrightarrow[x\to x_{i}]{}0.
\end{equation*}
Thus, for sufficiently small $\tau_{i}>0$ we have that
\begin{equation*}
	\frac{\|f(x)-Df(x_{i})(x-x_{i})\|}{\|x-x_{i}\|}<\frac{m}{t},
\end{equation*}
which contradicts \eqref{5.6} and ends the proof.
\end{proof}

On the other hand, since, by construction,  $f\vert_{B_{\eta_{i}}(x_{i})}$ is a diffeomorphism, necessarily  $Df(B_{\eta_{i}}(x_{i}))$ is a path connected subset of $GL(X,Y)$ and hence, since the restricted orientation
\begin{equation}\label{ll}
   \varepsilon|_{Df(B_{\eta_{i}}(x_{i}))}: \; Df(B_{\eta_{i}}(x_{i}))\cap GL(X,Y)\longrightarrow\mathbb{Z}_{2}
\end{equation}
is always constant in each path connected component of its domain, it is actually constant. Denote its constant value by $\varepsilon_{0}$.  Subsequently, we consider the map
\begin{equation*}
	\begin{array}{lccl}
	\varepsilon^{H_{i}}: & D_xH_i([0,1]\times B_{\eta_{i}}(x_{i}))& \longrightarrow & \mathbb{Z}_{2} \\
	 & T & \mapsto & \varepsilon_{0}
	\end{array}
\end{equation*}
Note that, thanks to \eqref{tata}, for each $t\in [0,1]$
$$
  D_xH_i(t,B_{\eta_{i}}(x_{i}))=t Df(B_{\eta_{i}}(x_{i}))+(1-t)Df(x_i)\subset GL(X,Y)
$$
is  a wedge in $GL(X,Y)$. Consequently, $D_xH_i([0,1]\times B_{\eta_{i}}(x_{i}))$ is a path connected subset of $GL(X,Y)$ and hence $\varepsilon^{H_{i}}$ provides us with an orientation of $H_{i}$.
Therefore, thanks to Lemmas \ref{le5.2} and \ref{le5.3}, it becomes apparent that $H_{i}$ is a proper $\mathcal{C}^{1}$-Fredholm homotopy with orientation $\varepsilon^{H_{i}}$ and $0\notin H_{i}([0,1]\times\partial B_{\tau_{i}}(x_{i}))$ for sufficiently small $\tau_i>0$. Moreover, $\varepsilon^{H_{i}}_{t}(\equiv\varepsilon_{0})$ provides us with an orientation of the section $H_{i}(t,\cdot)$ and therefore $(H_{i}(t,\cdot),\Omega,\varepsilon^{H_{i}}_{t})\in\mathscr{A}$ for each $t\in[0,1]$. By the axiom (H), and taking into account that $\varepsilon^{H_{i}}_{j}= \varepsilon|_{Df(B_{\eta_{i}}(x_{i}))}(\equiv\varepsilon_{0})$ for each $j\in\{0,1\}$,
\begin{equation}
\label{E1}	 \deg(f,B_{\tau_{i}}(x_{i}),\varepsilon)=\deg(Df(x_{i})(\cdot-x_{i}),B_{\tau_{i}}(x_{i}),\varepsilon).
\end{equation}
In order to remove the affine term in \eqref{E1}, we consider the homotopy
\begin{equation*}
	\begin{array}{lccl}
	G_{i}: & [0,1]\times \overline{\Pi} & \longrightarrow & Y \\
	 & (t,x) & \mapsto & Df(x_{i})(x-tx_{i})
	\end{array}
\end{equation*}
where $\Pi=\bigcup_{t\in[0,1]}B_{\tau_{i}}(tx_{i})$. Obviously, $G_{i}\in\mathcal{C}^{1}([0,1]\times \overline{\Pi},Y)$. Moreover, since, for every $t\in [0,1]$,  $G_{i}(t,\cdot)=Df(x_{i})(\cdot-t x_{i})$ is a diffeomorphism, we have that $G_{i}(t,\cdot)$ is proper for each $t\in[0,1]$. Therefore, since $G_{i}$ is uniformly continuous in $t$, it follows from Theorem 3.9.2 of \cite{ZR} that $G_{i}$ is proper.
\par
As $G_{i}(t,\cdot)=Df(x_{i})(\cdot-tx_{i})$ is a diffeomorphism for each $t\in [0,1]$ and $G_{i}(t,tx_{i})=0$, it is obvious that
$$
   0\notin G_{i}(t,\partial \Pi)\quad \hbox{for all} \;\; t\in[0,1].
$$
Moreover, since, for every $t\in [0,1]$,
$$
   D_x G_{i}(t,\cdot)=Df(x_{i})\in GL(X,Y),
$$
it is apparent that
$$
  D_{x}G_{i}([0,1]\times \Pi)=\{Df(x_{i})\}\subset GL(X,Y)
$$
and if we choose the orientation
$$
   \varepsilon^{G_{i}}: \{Df(x_{i})\}\longrightarrow\mathbb{Z}_{2}
$$
to be $\varepsilon^{G_{i}}\equiv \varepsilon_{0}$, $G_{i}$ is a proper $\mathcal{C}^{1}$-Fredholm homotopy with orientation $\varepsilon^{G_{i}}$ such that $0\notin G_{i}([0,1]\times\partial \Pi)$. Moreover, since $D_{x}G_{i}(\{t\}\times \Pi)\cap GL(X,Y)=\{Df(x_{i})\}\neq \emptyset$ for each $t\in[0,1]$, $\varepsilon^{G_{i}}_{t}(\equiv\varepsilon_{0})$ provides us with an orientation of the section $G_{i}(t,\cdot)$ and therefore $(G_{i}(t,\cdot),\Pi,\varepsilon^{G_{i}}_{t})\in\mathscr{A}$ for each $t\in[0,1]$. Thanks to the axiom (H) and taking into account that $\varepsilon^{G_{i}}_{j}=\varepsilon|_{Df(B_{\tau_{i}}(x_{i}))}(\equiv\varepsilon_{0})$ for each $j\in\{0,1\}$, we find that
\begin{align*}
\deg(Df(x_{i}),\Pi,\varepsilon) & = \deg(G_{i}(0,\cdot),\Pi,\varepsilon)=\deg(G_{i}(1,\cdot),\Pi,\varepsilon)\\ & =
\deg(Df(x_{i})(\cdot-x_{i}),\Pi,\varepsilon).
\end{align*}
Since
	$$
	F_{t}(x)=Df(x_{i})(x-tx_{i}), \quad x\in\Pi,
	$$
	is a diffeomorphism for each $t\in\{0,1\}$ and $F_{t}(tx_{i})=0$, we have that $0\notin F_{t}(\Pi\backslash B_{\tau_{i}}(tx_{i}))$. Thus, applying the axiom (A) with $\Omega=\Pi$, $\Omega_1=B_{\tau_i}(tx_i)$ and $\Omega_2=\emptyset$, it becomes apparent that
	\begin{equation*}
	\deg(Df(x_{i}),\Pi,\varepsilon)=\deg(Df(x_{i}),B_{\tau_{i}}(0),\varepsilon)
	\end{equation*}
	and
	\begin{equation*} \deg(Df(x_{i})(\cdot-x_{i}),\Pi,\varepsilon)=\deg(Df(x_{i})(\cdot-x_{i}),B_{\tau_{i}}(x_{i}),\varepsilon).
	\end{equation*}
	Therefore by the last argument,
$$
\deg(Df(x_{i})(\cdot-x_{i}),B_{\tau_{i}}(x_{i}),\varepsilon)=
\deg(Df(x_{i}),B_{\tau_{i}}(0),\varepsilon).
$$
Consequently, by the axiom (N)
\begin{equation}\label{E3}
\deg(Df(x_{i})(\cdot-x_{i}),B_{\tau_{i}}(x_{i}),\varepsilon)=
\deg(Df(x_{i}),B_{\tau_{i}}(0),\varepsilon)=\varepsilon(Df(x_{i})).
\end{equation}
Thus, combining \eqref{E1} and \eqref{E3} yields
\begin{equation*}
\deg(f,B_{\tau_{i}}(x_{i}),\varepsilon)=\varepsilon(Df(x_{i}))
\end{equation*}
and therefore, by \eqref{tete},
\begin{equation}
\label{titi}
\deg(f,\Omega,\varepsilon)=\sum_{i=1}^{n}\varepsilon(Df(x_{i})).
\end{equation}
Now, since $f^{-1}(0)\cap\Omega\neq \emptyset$ and $0$ is a regular value, necessarily $Df(\Omega)\cap GL(X,Y)\neq \emptyset$. Take $L\in Df(\Omega)\cap GL(X,Y)$. Then, according to \eqref{c.5},
$$
   \varepsilon(Df(x_{i}))=\varepsilon(L)\cdot\sigma(\mathfrak{L}_{x_{i}},[a,b]),
$$
where $\mathfrak{L}_{x_{i}}\in\mathcal{C}([a,b],Df(\Omega))$ is a Fredholm path linking $Df(x_{i})$ with $L$. By Theorem \ref{th3.3}, for any analytical curve $\mathfrak{L}_{\omega,x_{i}}\in\mathscr{C}^{\omega}([a,b],\Phi_{0}(X,Y))$ $\mathcal{A}$-homotopic to $\mathfrak{L}_{x_{i}}$
\begin{equation*}
\varepsilon(Df(x_{i}))=\varepsilon(L)\cdot\sigma(\mathfrak{L}_{x_{i}},[a,b])=\varepsilon(L)\cdot
(-1)^{\sum_{\lambda_{x}\in\Sigma(\mathfrak{L}_{\omega,x_{i}})}\chi[\mathfrak{L}_{\omega,x_{i}};\lambda_{x}]}.
\end{equation*}
Therefore, by \eqref{titi},
\begin{equation*}
\deg(f,\Omega,\varepsilon)=\varepsilon(L)\cdot\sum_{i=1}^{n} (-1)^{\sum_{\lambda_{x}\in\Sigma(\mathfrak{L}_{\omega,x_{i}})}
\chi[\mathfrak{L}_{\omega,x_{i}};\lambda_{x}]},
\end{equation*}
which ends the proof of the theorem in the regular case. We have actually proven that in the regular case
any map satisfying the axioms (N), (A) and (N) must coincide with the degree of  Fitzpatrick, Pejsachowicz and Rabier \cite{FPR}.
\par
If the general case when $(f,\Omega,\varepsilon)\notin\mathscr{R}$, for every $\eta>0$, by a theorem of  Quinn and Sard, \cite{QS}, there exists a regular value $x_{0}$ such that $\|x_{0}\|<\eta$. Let $H$ be the homotopy defined by
\begin{equation*}
\begin{array}{lccl}
H: & [0,1]\times\overline{\Omega} & \longrightarrow & Y \\
 & (t,x) & \mapsto & f(x)-tx_{0}
\end{array}
\end{equation*}
Then, $H\in\mathcal{C}^{1}([0,1]\times\overline{\Omega},Y)$ and it is proper. Obviously,
$H\in\mathcal{C}^{1}([0,1]\times\overline{\Omega},Y)$ and
$$
   D_xH(t,\cdot)=Df(\cdot)\in \Phi_{0}(X,Y).
$$
First, we will prove that $H(t,\cdot)$ is proper for each $t\in[0,1]$. By the definition of $H$,
for any compact subset, $K$, of $Y$, 	
\begin{equation*}
	H(t,\cdot)^{-1}(K)=f^{-1}(K+tx_{0})
\end{equation*}
is compact because  $K+tx_{0}$ is compact and $f$ proper. Therefore, as the map $H$ is uniformly continuous in $t$, as above, it follows from Theorem 3.9.2 of \cite{ZR} that $H$ is proper.
\par
Now, we will show that $0\notin H(t,\partial\Omega)$ for each $t\in[0,1]$. On the contrary, suppose that $0\in H(t,\partial\Omega)$ for some $t\in[0,1]$. Then, there exists $x\in\partial\Omega$ such that $f(x)=tx_{0}$. In particular, $tx_{0}\in f(\partial\Omega)$. Since $f$ is proper, by Lemma 3.9.1 of \cite{ZR}, $f$ is a closed map and since $\partial\Omega$ is closed, $f(\partial\Omega)$ is closed. Since $0\notin f(\partial\Omega)$ and $f(\partial\Omega)$ is closed, there exists $\eta>0$ such that $B_{\eta}(0)\cap f(\partial\Omega)=\emptyset$. As we have already taken $\|x_{0}\|<\eta$, we also have that $tx_{0}\in B_{\eta}(0)$ and therefore $tx_{0}\notin f(\partial\Omega)$. This contradicts $tx_{0}\in f(\partial\Omega)$.
\par
Since $D_xH(\{t\}\times\Omega)=Df(\Omega)$ for each $t\in[0,1]$, $D_{x}H([0,1]\times\Omega)=Df(\Omega)$ and if we define $\varepsilon^{H}: Df(\Omega)\cap GL(X,Y)\to\mathbb{Z}_{2}$ by $\varepsilon^{H}\equiv\varepsilon$, where $\varepsilon$ is the orientation of $Df(\Omega)$ given by \eqref{ll}, then $H$ is a proper $\mathcal{C}^{1}$-Fredholm homotopy with orientation $\varepsilon^{H}$ and $0\notin H([0,1]\times \partial\Omega)$. Observe that in this case the domain is the whole $\Omega$ and therefore the orientation $\varepsilon$ might not be, in general, constant. Moreover, since $D_xH(\{t\}\times\Omega)=Df(\Omega)$ and $Df(\Omega)$ is orientable with orientation $\varepsilon(=\varepsilon^{H}_{t})$ for each $t\in[0,1]$, necessarily $(H(t,\cdot),\Omega,\varepsilon^{H}_{t})\in\mathscr{A}$ for each $t\in[0,1]$. Owing to the axiom (H) and taking into account that $\varepsilon^{H}_{j}=\varepsilon$ for each $j\in\{0,1\}$
\begin{equation*}
\deg(f,\Omega,\varepsilon)=\deg(f-x_{0},\Omega,\varepsilon).
\end{equation*}

\noindent Since $x_{0}$ is a regular value of $f$, we have that $0$ is a regular value of $f-x_{0}$ and $(f-x_{0},\Omega,\varepsilon)\in\mathscr{R}$. This reduces the case to the regular one discussed previously and proves that the map $\deg:\mathscr{A}\to\mathbb{Z}$ coincides with the one constructed by Fitzpatrick, Pejsachowicz and Rabier \cite{FPR}. This concludes the proof.
\end{proof}

\end{document}